\documentclass[10pt,a4paper]{article}
\usepackage[a4paper,margin=3cm]{geometry}

\usepackage[T1]{fontenc}
\usepackage[utf8]{inputenc}
\usepackage{lmodern}
\usepackage{csquotes}
\usepackage{amsmath}
\usepackage{amsthm}
\usepackage{dsfont}
\usepackage{mathtools}
\usepackage{amssymb}
\usepackage{amsfonts}
\usepackage{mathrsfs}
\usepackage{graphicx}
\usepackage[shortlabels]{enumitem}
\usepackage{xcolor}
\usepackage{authblk}
\usepackage{hyperref}

\definecolor{darkblue}{rgb}{0,0,0.5}
\definecolor{darkred}{rgb}{0.5,0,0}
\definecolor{darkgreen}{rgb}{0,0.5,0}

\hypersetup{
  unicode=true,          
  pdffitwindow=false,     
  pdfstartview={FitH},    
  colorlinks=true,       
  linkcolor=darkblue,          
  citecolor=darkred,        
  filecolor=magenta,      
  urlcolor=darkgreen           
}

\usepackage[scaled=1,sups]{XCharter}
\usepackage[charter,varbb,smallerops,scaled=1.08]{newtxmath}
\usepackage[numbers,sort&compress]{natbib}
\usepackage{doi}

\usepackage{parskip}
\begingroup
\makeatletter
\@for\theoremstyle:=definition,remark,plain\do{%
  \expandafter\g@addto@macro\csname th@\theoremstyle\endcsname{%
    \addtolength\thm@preskip\parskip }%
}
\endgroup

\theoremstyle{plain}
\newtheorem{theorem}{Theorem}
\newtheorem{corollary}[theorem]{Corollary}
\newtheorem{prop}[theorem]{Proposition}
\newtheorem{lemma}[theorem]{Lemma}

\theoremstyle{definition}
\newtheorem{definition}[theorem]{Definition}

\newtheorem{remark}[theorem]{Remark}

\newcommand{\R}{\mathbb{R}}
\newcommand{\N}{\mathbb{N}}
\renewcommand{\P}{\mathbb{P}}
\newcommand{\E}{\mathbb{E}}
\newcommand{\1}{\mathds{1}}
\newcommand{\ed}{\overset{(d)}{=}}

\renewcommand{\d}{\mathrm{d}}
\newcommand{\e}{\mathrm{e}}

\newcommand{\norm}[1]{\lVert #1 \rVert}
\newcommand{\abs}[1]{\lvert #1 \rvert}

\renewcommand{\epsilon}{\varepsilon}
\renewcommand{\rho}{\varrho}
\renewcommand{\phi}{\varphi}
\renewcommand{\emptyset}{\varnothing}

\newcommand{\partit}[1]{\mathcal{P}_{#1}}
\newcommand{\markpart}[1]{\mathcal{M}_{#1}}
\newcommand{\markpartstar}{\mathcal{M}^{\star}_{\infty}}
\newcommand{\masspart}{\mathscr{S}^\downarrow}
\newcommand{\markmasspart}{\mathscr{Z}^\downarrow}
\newcommand{\restr}[1]{_{|[#1]}}

\DeclareMathOperator{\Exp}{Exp}
\DeclareMathOperator{\Be}{Be}
\DeclareMathOperator{\frag}{Frag}
\DeclareMathOperator{\ssfrag}{ssFrag}
\DeclareMathOperator*{\tol}{\longrightarrow}

\title{Fragmentations with self-similar branching speeds}
\author{Jean-Jil \textsc{Duchamps}}
\date{\today}

\affil{LmB UMR 6623, Université Bourgogne Franche-Comté, CNRS, F-25000 Besançon, France}

\begin{document}
\maketitle

\begin{abstract}
  We consider fragmentation processes with values in the space of marked partitions of $\N$, i.e.\ partitions where each block is decorated with a nonnegative real number.
  Assuming that the marks on distinct blocks evolve as independent positive self-similar Markov processes and determine the speed at which their blocks fragment, we get a natural generalization of the self-similar fragmentations of~\cite{Ber02}.
  Our main result is the characterization of these generalized fragmentation processes: a Lévy-Khinchin representation is obtained, using techniques from positive self-similar Markov processes and from classical fragmentation processes.
  We then give sufficient conditions for their absorption in finite time to a frozen state, and for the genealogical tree of the process to have finite total length.
\end{abstract}

\begin{table}[b!]
  \small
  \rule{.5\linewidth}{.4pt}
  
  \smallskip
  
\end{table}

\tableofcontents


\section{Introduction}

A fragmentation process is a system of particles evolving in time in a Markovian way, where each particle is assigned a mass and may dislocate at random times, distributing its mass among newly created particles.
It is usually assumed that particles evolve independently of one another, in a way depending only on their mass.
Self-similar fragmentations are processes where the speed of fragmentation of a particle is accelerated proportionally to a function of its mass -- which then must be a power function, characterized by an exponent $\alpha\in\R$.
These processes are said to be homogeneous when $\alpha=0$.
Homogeneous and self-similar fragmentations have been characterized in the early 2000s (see \cite{Ber02b,Ber02}, or \cite{Ber06} for a general introduction), and their connections to random trees have been developed in e.g.\ \cite{AP98} or \cite{HM04,HMPW08}.

These studies have been made under a conservative assumption, which prevents the total mass in the system from increasing.
This assumption allows for instance the representation of fragmentation processes in terms of exchangeable partition-valued processes, which are convenient objects allowing one to naturally recover \emph{discrete} genealogical structures in fragmentation processes.

The primary goal of this article is to extend the self-similar assumption while staying in a conservative setting.
To this aim, we assume that particles are described by a pair mass-mark which evolves jointly in a Markovian way, such that a) the total mass does not increase, and b) it is now the \emph{mark} -- which may a priori fluctuate in any way -- of a particle which determines the speed at which it fragments.
The conservative assumption allows us to model this idea with Markov processes taking values in marked partitions of the integers, with very little restriction concerning marks.
Consequently, if one ignores the masses of particles, our processes essentially give constructions for quite general non-conservative fragmentations.
Related and inspiring works include self-similar branching Markov chains~\cite{Kre09}, the recent so-called branching Lévy processes of~\cite{BM18}, as well as many recent developments which have been published on self-similar growth-fragmentation processes (see e.g.\ \cite{Dad17,Ged19,Shi17}), introduced by Bertoin~\cite{Ber17}, which allow masses of particles to fluctuate as a positive Markov process.
Note the difference between our processes and the so-called (self-similar) multi-type fragmentation processes \cite{Ber08,Ste18,HS19+}, for which marks take a finite number of values but determine more than just the speed at which particles undergo fragmentation.
In our setting, marks live in a continuous space but two versions of a fragmentation started from a single particle, with different initial marks, will have essentially the same distribution up to a time-change -- which is incorrect in general for multi-type fragmentations.

The article is organized as follows.
In the remainder of the introduction, we recall some definitions and basic results of usual self-similar fragmentations, and define the space of marked partitions in which our processes live.
In Section~\ref{sec:ESSF-first-props} we define our extended self-similar fragmentation (ESSF) processes, and point out their basic properties.
We characterize ESSF processes with a type of Lévy-Khinchin representation in Section~\ref{sec:main-res}, and then give sufficient conditions for a process to almost surely a) reach an absorbing state in finite time b) have a genealogy where the sum of lengths of all branches is finite.
Because most proofs are somewhat technical, we defer them to Appendix~\ref{sec:app} to ease the exposition.

\subsection{Self-similar fragmentations}
To study processes with values in the space of partitions of $\N$, let us recall some classical notation and definitions.
First define $[n]:=\{1,2,\ldots, n\}$ for $n\in\N$ and $[\infty]:=\N=\{1,2,\ldots\}$.
Now for $n\in \N\cup\{\infty\}$, we denote by $\partit{n}$ the space of partitions of $[n]$.
We often see a partition $\pi\in\partit{n}$ as the equivalence relation $\sim^{\pi}$ it represents on $[n]$.
We will denote by $\mathbf{0}_{n}$ (resp.\ $\mathbf{1}_{n}$) the partition of $[n]$ into singletons (resp.\ the partition with a single block $\{[n]\}$).
We will often omit the subscript $n$ and write only $\mathbf{0}$ or $\mathbf{1}$ when the context is clear.

For $n < m \leq \infty$ and $\pi\in\partit{m}$, we denote by $\pi\restr{n}$ its restriction to the set $[n]\subset[m]$.
$\partit{\infty}$ may be understood as the projective limit of the sets $(\partit{n},n\in\N)$, and as such, a natural metric which makes this space compact may be defined on it by
\[
d(\pi,\pi') = \sup\{n\in\N,\,\pi\restr{n}=\pi'\restr{n}\}^{-1},
\]
where by convention $(\sup \N)^{-1}=0$.
We will consider the action of permutations of $\N$ on $\partit{\infty}$, and more generally we can define, for any $1\leq n \leq m \leq \infty$, any \emph{injection} $\sigma: [n]\to[m]$ and any $\pi\in\partit{m}$, the partition $\pi^{\sigma}\in\partit{n}$ defined by:
\[
i\sim^{\pi^{\sigma}}j \iff \sigma(i)\sim^{\pi}\sigma(j), \qquad i,j\in[n].
\]
Note that in this paper, a permutation $\sigma:\N\to\N$ is a bijection with \emph{finite} support ${\{n\in\N,\,\sigma(n)\neq n\}}$.
We usually label the blocks of a partition $\pi=\{\pi_1,\pi_2,\ldots\}$ in the unique way such that the sequence
$(\min \pi_k, k\geq 1)$ is increasing.
This way, $\pi_1$ is necessarily the block containing $1$, $\pi_2$ is the block containing the lowest integer not in the same block as $1$, etc.
By convention, if $\pi$ has a finite number of blocks, say $K$, we define $\pi_{K+l}=\emptyset$ for all $l\geq 1$.
It will be useful to define a fragmentation operator $\frag:\partit{\infty}\times(\partit{\infty})^{\N} \to \partit{\infty}$ by
\[
\frag(\pi,\pi^{(\cdot)}) = \{\pi_k\cap \pi^{(k)}_l, \; k,l\geq 1\},
\]
where $(\pi_k)$ are the ordered blocks of $\pi$ and $(\pi^{(k)}_l)$ the ordered blocks of $\pi^{(k)}$.
In words, blocks of the new partition are formed from the restriction of the $k$-th partition of the sequence $\pi^{(\cdot)}$ to $\pi_k$, for each $k\geq 1$.

Now let us recall the definition of partition-valued fragmentation processes (see e.g.\ \cite{Ber06}).
For this definition, we restrict ourselves to the space of partitions that have asymptotic frequencies, i.e.\ $\pi\in\partit{\infty}$ such that for all $k\geq 1$,
\[
\abs{\pi_k} := \lim_{n\to\infty}\frac{\# \pi_k\cap[n]}{n} \quad \text{exists}.
\]
In this case, we write $\abs{\pi}^{\downarrow}$ for the nonincreasing reordering of the sequence $(\abs{\pi_1},\abs{\pi_2},\ldots)$.
Let us write $\partit{\infty}'$ for the space of partitions of $\N$ with asymptotic frequencies.

\begin{definition} \label{def:ssfrag-usual}
  A self-similar fragmentation process is a càdlàg Markov process $(\Pi(t),t\geq 0)$ with values in $\partit{\infty}'$, such that almost surely for all $k\in\N$, the map $t\mapsto\abs{\Pi_k(t)}$ is right-continuous and for which the following properties hold.
  \begin{enumerate}[(i)]
    \item \emph{Exchangeability}: for all $\pi\in\partit{\infty}'$, for all permutations $\sigma:\N\to\N$,
    \begin{equation*}
    (\Pi(t)^{\sigma}, \, t\geq 0) \text{ under }\P_{\pi} \;\ed\;  (\Pi(t), \, t\geq 0) \text{ under }\P_{\pi^{\sigma}},
    \end{equation*}
    where $\P_{\pi}$ denotes the distribution of the process started from $\pi$.
    
    \item \emph{Self-similar branching}: there exists $\alpha\in\R$ such that if $(\Omega,\P)$ is a probability space where $(\Pi^{(\cdot)}(t),t\geq 0)$ is a sequence of independent copies of the process started from $\mathbf{1}$, then for any $\pi\in\partit{\infty}'$, we have
    \begin{equation}\label{eq:ssf_branching}
    (\Pi(t), \, t\geq 0) \text{ under }\P_{\pi} \;\ed\; (\frag(\pi,\widetilde{\Pi}^{(\cdot)}(t)),\,t\geq 0) \text{ under } \P,
    \end{equation}
    where $\widetilde{\Pi}^{(\cdot)}$ is the sequence of time-changed processes defined by
    \[
    \widetilde{\Pi}^{(k)}(t) = \Pi^{(k)}(\abs{\pi_k}^{\alpha}t), \qquad k\geq 1, t\geq 0.
    \]
  \end{enumerate}
\end{definition}
Note that a fragmentation with self-similarity index $\alpha =0$ is called homogeneous.
It is well-known (we refer to \cite[Section 1 to 3]{Ber06} for a detailed account on the theory of partition-valued fragmentations) that self-similar fragmentations can be characterized in terms of their self-similarity index $\alpha$, a so-called erosion coefficient $c\geq 0$ and a dislocation measure $\nu$ on the (metric and compact when equipped with the uniform distance) space 
\[
\masspart:=\big\{\mathbf{s} = (s_1,s_2,\ldots)\in[0,1]^{\N}\text{ where }s_1 \geq s_2 \geq \ldots \geq 0\text{ and }\sum_{k}s_k\leq 1\big\},
\]
satisfying
\[ \int_{\masspart}(1-s_1)\,\nu(\d \mathbf{s}) < \infty. \]
In words, $c$ is the rate at which each singleton detaches from ``macroscopic'' blocks and $\nu$ is a measure giving the rates of ``sudden dislocations'', i.e.\ a block with asymptotic frequency $x$ fragments at rate $\nu(\d \mathbf{s})$ into (possibly infinitely many) blocks with frequencies given by $x\mathbf{s} = (xs_1,xs_2, \ldots)$ -- these dislocations of blocks are usually represented by a so-called paintbox process, which we will define in the context of marked partitions in the next section.
The self-similarity index $\alpha$ of a fragmentation encodes, through property \eqref{eq:ssf_branching}, the speed at which blocks fragment, depending on their size.
For instance, if $\alpha$ is negative, then there is a random time $T$ which is finite almost surely at which $\Pi(T)$ is the partition into singletons, whereas it is never the case when $\alpha\geq 0$ and $\nu(s_1=0)=0$.
Note that $\alpha=0$ means that there is no time change -- in that case the sequence $\widetilde{\Pi}^{(\cdot)}$ in \eqref{eq:ssf_branching} is simply $\Pi^{(\cdot)}$ -- the process is then said to be homogeneous.

Our goal is to generalize these objects and define processes $(\Pi(t), \mathbf{V}(t), t\geq 0)$, where $\Pi$ is partition-valued and $\mathbf{V}(t)=(V_n(t),n\geq 1)$ is a random map $\N\to [0,\infty)$ playing the role of $(\abs{\pi_k}^{\alpha},k\geq 1)$, i.e.\ dictating the speed of fragmentation of different blocks of $\Pi$.
To define this we need first to introduce the formalism of marked partitions and processes in this space.

\subsection{Partitions with marks} \label{sec:paintbox}

Let us consider partitions where each block is decorated with a mark.
For convenience, we consider that the space of marks is the space $[0,\infty]$ where $0$ is identified with $\infty$.
Formally we define this as the set $[0,\infty)$ endowed with the topology consisting of a) the usual open sets of $(0,\infty)$ and b) the open sets of $[0,\infty)$ containing $0$ \emph{and} a half-line $(A,\infty)$ for some $A > 0$.
Topologically it is a circle so we will denote it by $S^1$, but throughout the paper, elements of $S^1$ will be identified with their unique representative in $[0,\infty)$,
so that $S^1$ is furthermore endowed with a multiplicative structure, inherited from that of $[0,\infty)$.
This enables us to consider for instance the maps
\[
m_x : v \mapsto xv \qquad \text{and} \qquad p_\alpha: v\mapsto v^{\alpha}
\]
as well-defined and continuous on $S^1$, where $x$ is in $S^1$ or $[0,\infty)$ and $\alpha\in\R\setminus\{0\}$.
Note that for a technical reason, we choose to use throughout the article the convention $0^0=0$, so that $0^{\alpha}=0$ for any $\alpha\in\R$, and $v^0=\1_{v\neq 0}$ for any $v\in S^1$.

For $n\in\N\cup\{\infty\}$, we consider the space of marked partitions defined by
\begin{align*}
\markpart{n} &:= \big\{x = (\pi,\mathbf{v}) \in \partit{n}\times (S^1)^{[n]}, \; 
v \text{ is constant on the blocks of }\pi
\big\}\\
& \;=
\bigcap_{(i,j) \in [n]^2}\big\{x = (\pi,\mathbf{v}) \in \partit{n}\times (S^1)^{[n]}, \; i\nsim^{\pi}j \text{ or } v_i=v_j \big\}.
\end{align*}
It is a closed subset of $\partit{n}\times(S^1)^{[n]}$, which, endowed with the product topology, is compact metrizable, therefore Polish.
Note that by definition, if $(\pi,\mathbf{v})\in\markpart{n}$ where $\pi=\mathbf{1}$ is the partition into a single block, then $\mathbf{v}$ is of the form $(v,v,\ldots)$ for a unique $v\in S^1$.
For this reason we will use the abuse of notation $(\mathbf{1},v)$ to denote this element.
We see $x=(\pi,\mathbf{v})$ as the partition $\pi$ where each block is given a mark.
Therefore, we will sometimes say $B$ is \emph{a block of $x$ with mark $v$} if $B\in \pi$ and $v_i=v$ for some (therefore all) $i\in B$.
Similarly, we will use the notation $i\sim^{x}j$ if $i$ and $j$ are in the same block of~$\pi$.

Note that for $n < m \leq \infty$ and $x=(\pi,\mathbf{v})\in\markpart{m}$, we can naturally consider the restrictions $x\restr{n} = (\pi,\mathbf{v})\restr{n} := (\pi\restr{n},(v_1, v_2, \ldots, v_n)) \in \markpart{n}$, which are clearly continuous maps.

Similarly, we can extend the action of injections $\sigma:[n]\to[m]$ to our context and define for $x=(\pi,\mathbf{v})\in\markpart{m}$,
\[
x^{\sigma}=(\pi,\mathbf{v})^{\sigma} = (\pi^{\sigma},\mathbf{v}^{\sigma}) := (\pi^{\sigma},(v_{\sigma(i)},\,i\in [n])) \in\markpart{n}.
\]
We say that a random variable $X$ with values in $\markpart{\infty}$ is exchangeable if for all permutations ${\sigma:\N\to\N}$,
\[
X^{\sigma} \ed X.
\]
Finally we can also extend the fragmentation operator $\frag$ to marked partitions by setting
\[
\frag\!\big((\pi,\mathbf{v}),(\pi^{(\cdot)},\mathbf{v}^{(\cdot)})\big) := \big(\!\frag(\pi,\pi^{(\cdot)}), \widetilde{\mathbf{v}}\big),
\]
where, for $i\geq 1$, $\widetilde{v}_{i}$ is defined by $v_i {v_i}^{(k_i)}$, where $k_i$ is the label of the block of $\pi$ containing~$i$ -- so that $i$ is in the $k_i$-th block of $\pi$.

We say that a marked partition $x\in\markpart{\infty}$ is \emph{non-degenerate} if every finite block has mark $0$, and we denote the space of non-degenerate marked partitions by
\[
\markpartstar := \big\{x=(\pi,\mathbf{v})\in \markpart{\infty}, \; \forall i\geq 1,\, 
\text{if }i\text{ is in a finite block of }\pi\text{ then }v_i=0
\big\}.
\]
In particular for singleton blocks, $\{i\}\in\pi$ implies $v_i=0$.
Note that this space is still Polish \cite[see e.g.][Theorem 2.2.1]{Sri98} as a $G_\delta$-subset -- a countable intersection of open sets -- of $\markpart{\infty}$.
Indeed, letting for all $i\in\N$, $N_i:\markpart{\infty}\to\N\cup\{\infty\}$ be the map associating $(\pi,\mathbf{v})$ with the cardinality of the block of $\pi$ containing $i$, then, taking $d$ to be any metric on $S^1$ compatible with its topology, we can write
\begin{align*}
\markpartstar &= \bigcap_{i\geq 1}\{N_i < \infty \implies v_i = 0\}\\
&= \bigcap_{i,j,k\geq 1}\big(\{N_i \geq j\}\cup\{d(v_i,0) < 1/k\}\big),
\end{align*}
which is a countable intersection of open subsets of $\markpart{\infty}$.
Note that if $n$ is finite, one cannot define an analogous property of non-degeneracy for marked partitions in $\markpart{n}$.

Now let us define paintbox processes for exchangeable marked partitions.
Consider the space $([0,1]\times[0,\infty),\preceq)$ equipped with the lexicographic order, that is if $z=(s,v)\in [0,1]\times[0,\infty)$ and $z'=(s',v')\in[0,1]\times[0,\infty)$, then
\[
z\preceq z' \iff s < s' \text{ or }(s=s'\text{ and }v\leq v').
\]
Let us define
-- using the notation $z_k=(s_k,v_k)$ --
\[
\markmasspart_0 := \big\{ \mathbf{z}=(z_1,z_2,\ldots) \in \big([0,1] \times [0,\infty)\big)^{\N},\; z_1 \succeq z_2 \succeq \ldots,\text{ and }\sum_k s_k \leq 1 \big\},
\]
and note that, endowed with the product topology, it is a Polish space.
Indeed, it can be written
\[
\markmasspart_0 = \big\{\sum_k s_k \leq 1 \big\}\cap \bigcap_{i\geq 1}\Big(\big\{ s_i > s_{i+1}\big\}\cup\big\{ s_i = s_{i+1} \text{ and } v_i \geq v_{i+1} \big\}\Big),
\]
which is a countable intersection of closed and open subsets of $\big([0,1] \times [0,\infty)\big)^{\N}$.
This space being Polish, closed sets are $G_{\delta}$, and so $\markmasspart_0$ is Polish.
Because this will be consistent with our previous definition of $\markpartstar$, we want to ignore the possible indices $k\geq 1$ such that $s_k=0$.
Therefore, we will rather use the space
\begin{align*}
\markmasspart &:=  \big\{ \mathbf{z}\in \markmasspart_0,\; \forall k \geq 1,\, s_k=0\implies v_k=0 \big\}\\
&= \bigcap_{k,l \geq 1}\big\{\mathbf{z}\in\markmasspart_0, \,  s_k > 0 \text{ or } v_k < 1/l\big\},
\end{align*}
which is still Polish.

Similarly as in the usual case, we say that $x=(\pi,\mathbf{v})\in\markpart{\infty}$ has asymptotic frequencies if $\pi$ has asymptotic frequencies.
In that case, we define $\abs{x}^{\downarrow}\in \markmasspart$ as the nonincreasing reordering (with respect to the lexicographic order $\preceq$ on $[0,1]\times[0,\infty)$) of the sequence of pairs
\[
\big((\abs{B_i},v_i),\, B_i \text{ is the }i\text{-th infinite block of }x\text{ such that }\abs{B_i}>0\text{ and with mark }v_i\big).
\]
Note that we consider only blocks $B$ satisfying $\abs{B}>0$ in the previous display since in general the set $\{(\abs{B_i},v_i),\, B_i \text{ is the }i\text{-th block of }x\text{, with mark }v_i\}$ may be impossible to enumerate in nonincreasing order.

Now let us introduce a paintbox construction for marked partitions.
Consider $\mathbf{z}=(\mathbf{s},\mathbf{v})\in\markmasspart$, and let $(U_n, n\geq 1)$ be an i.i.d.\ sequence of $[0,1]$-uniform random variables.
Define $X=(\Pi,\mathbf{V})$ as the $\markpartstar$-valued random variable given by the following relation:
let $t_n := \sum_{k=1}^{n} s_k$, with $t_0 := 0$ by convention, and
\begin{gather*}
i\sim^{\Pi}j \iff i=j \text{ or }\exists n\geq 1, \; t_{n-1} \leq U_i, U_j < t_n,\\
V_i := 
\begin{cases}
v_n &\text{ if }t_{n-1} \leq U_i < t_n, \text{ for }n\geq 1,\\
0 	&\text{ if } \sum_ks_k \leq U_i.
\end{cases}
\end{gather*}
It is easily checked that the random variable $X$ is exchangeable.
Also, recall the definition of asymptotic frequencies for a marked partition, and note that the law of large numbers implies $\abs{X}^{\downarrow}=\mathbf{z}$ almost surely.
We denote by $\rho_{\mathbf{z}}$ the distribution of $X$.
We will also make use of the distribution of $X\restr{n}$ for $n\in\N$, which we denote by $\rho^n_{\mathbf{z}}$.
Note that for any $v\in S^1$, if $\mathbf{z}=(\mathbf{s},\mathbf{v})\in\markmasspart$ is the unique element such that $s_1 = 1$ and $v_1 = v$, then $\rho_{\mathbf{z}} = \delta_{(\mathbf{1},v)}$.
For this reason, we will again abuse notation and let $(\mathbf{1},v)\in \markmasspart$ denote this element, so that $\rho_{(\mathbf{1},v)} = \delta_{(\mathbf{1},v)}$.

It is well-known since the work of Kingman \cite[Theorem 2]{Kin82} that the law of an exchangeable partition can be expressed as a mixture of paintbox processes.
Using similar arguments, one obtains the following result for marked partitions.

\begin{prop} \label{prop:paintbox}
  Let $X$ be an exchangeable random variable with values in $\markpartstar$.
  Then there exists a unique probability measure $\nu$ on $\markmasspart$ such that
  \begin{equation} \label{eqprop:paintbox}
  \P\big(X\in\cdot\big)=\int_{\markmasspart}\rho_{\mathbf{z}}(\cdot)\,\nu(\d \mathbf{z}).
  \end{equation}
\end{prop}

\begin{proof}
  See Appendix \ref{sec:proof-paintbox}.
\end{proof}

This setting of marked partitions being in place, we can now define our objects of study.


\section{Extended self-similar fragmentations} \label{sec:ESSF-first-props}

\subsection{Definitions, first properties}

Let us now define self-similar fragmentation processes with values in $\markpart{\infty}$.
For this, let us introduce a family of \emph{self-similar fragmentation} operators $(\ssfrag_\alpha, \alpha\in\R)$, defined as follows.
For $n\in\N\cup\{\infty\}$, consider a marked partition $x=(\pi,\mathbf{v})\in\markpart{n}$ and a sequence $\bar{x}^{(\cdot)}$ of càdlàg \emph{maps} ${\bar{x}^{(k)}\colon[0,\infty)\to\markpart{n}}$, satisfying $\bar{x}^{(k)}(0)=(\mathbf{1},1)$.
Writing for all $k\geq 1$ and $t\geq 0$, $\bar{x}^{(k)}(t)=(\bar\pi^{(k)}(t),\bar {\mathbf{v}}^{(k)}(t))$, we define
\begin{equation*}\ssfrag_\alpha(x,\bar{x}^{(\cdot)}):=(\hat{\pi},\hat{\mathbf{v}})
\end{equation*}
as the map $[0,\infty)\to\markpart{n}$ such that
\begin{gather*}
\hat{v}_i(t) = v_i \bar{v}^{(k_i)}(v_i^{\alpha} t) \\
i\sim^{\hat{\pi}(t)}j \iff i\sim^{\pi}j \;\text{ and }\;i\sim j\text{ in }\bar{\pi}^{(k_i)}(v_i^{\alpha}t),
\end{gather*}
where $k_i$ is defined as the label of the block of $\pi$ containing~$i$ (i.e.\ such that $i$ is in the $k_i$-th block of~$\pi$).
Note that
we use the convention $0^\alpha = 0$ for all $\alpha\in\R$, so that
thanks to this definition, if a block $B$ of $x$ has mark $0$, then the process $\ssfrag_\alpha(x,\bar{x}^{(\cdot)})$ is frozen at block $B$, in the sense that for all $t\geq 0$, $B$ is a block of $\hat{\pi}(t)$, and every $j\in B$ will have $\hat{v}_j(t)=0$.
Also, the assumptions on the maps $\bar{x}^{(k)}$ imply that $\ssfrag_\alpha(x,\bar{x}^{(\cdot)})$ is càdlàg and satisfies $\ssfrag_\alpha(x,\bar{x}^{(\cdot)})(0)=x$.

\begin{remark}~\label{rk:fragproc_continuity}
  \begin{enumerate}[(i)]  
    \item Consider here a convergent sequence $x_n=(\pi_{n},\mathbf{v}_{n})\to x=(\pi,\mathbf{v})\in\markpart{\infty}$, and assume that $v_{n,i} = 0$ for all $n\geq 1$ whenever $v_i=0$ for some $i$.
    If additionally we have for some $t\geq 0$, for all $i\geq 1$ such that $v_i>0$, and for all $k\geq 1$,
    \[
    \bar{x}^{(k)}(v^{\alpha}_{n,i}t) \underset{n\to\infty}{\longrightarrow} \bar{x}^{(k)}(v^{\alpha}_{i}t) ,
    \]
    then it is a straightforward consequence of the definition that
    \[
    \ssfrag_\alpha(x_{n},\bar{x}^{(\cdot)})(t) \underset{n\to\infty}{\longrightarrow} \ssfrag_{\alpha}(x,\bar{x}^{(\cdot)})(t).
    \]
    
    \item Note that one could define $\ssfrag_\alpha$ in terms of $\frag$ because we have the equality
    \[
    \ssfrag_\alpha\!\big(x,\bar{x}^{(\cdot)}\big)(t) = \frag\!\big(x, \bar{x}^{(\cdot)}(w^{\alpha}_{(\cdot)}t)\big),
    \]
    where $w_{(\cdot)}$ is the vector defined by $w_{(k)}=v_i$, for any $i$ in the $k$-th block of $\pi$.
  \end{enumerate}
\end{remark}

We can now define the following generalization of self-similar fragmentations.

\begin{definition} \label{def:ESSF}
  Let $X(t)=(\Pi(t), \mathbf{V}(t), t\geq 0)$ be a stochastic process with values in $\markpart{\infty}$.
  We say that $X$ is an \emph{extended self-similar fragmentation} (ESSF) process if it is a \emph{stochastically continuous} \emph{strong Markov} process with \emph{càdlàg sample paths}, for which the following properties hold:
  \begin{enumerate}[(i)]
    \item \emph{Exchangeability:}  for all permutations $\sigma:\N\to\N$, for all $x\in\markpart{\infty}$,
    \[
    (X(t)^{\sigma}, \, t\geq 0) \text{ under }\P_{x} \ed  (X(t), \, t\geq 0) \text{ under }\P_{x^{\sigma}} ,
    \]
    where $\P_{x}$ denotes the distribution of the Markov process started from $x$.
    
    \item \emph{Self-similar branching:} there exists $\alpha\in\R$ such that for all $x\in\markpart{\infty}$,
    \[
    X \text{ under }\P_{x} \ed \ssfrag_\alpha\!\big(x, X^{(\cdot)}\big),
    \]
    where $X^{(\cdot)}$    is an i.i.d.\ sequence of copies of the process started from $(\mathbf{1},1)$.
    As usual, we call $\alpha$ the index of self-similarity, and we will say for conciseness that $X$ is an $\alpha$-ESSF.
    For the special case $\alpha=0$, we will sometimes say the process $X$ is \emph{homogeneous}.
  \end{enumerate}
  An ESSF process $X$ will be called \emph{non-degenerate} if for all $x\in\markpartstar$, the process has sample paths in $\markpartstar$, $\P_{x}$-almost surely.
\end{definition}

\begin{remark}~
  \begin{enumerate}[(i)]
    \item 
    Consider $X=(\Pi,\mathbf{V})$ an $\alpha$-ESSF and $\gamma\in\R\setminus\{0\}$.
    Let us define $Y:=(\Pi,\mathbf{V}^{\gamma})$, where $\mathbf{V}^{\gamma}(t)$ is simply the vector $(V_i(t)^{\gamma},\,i\geq 1)$.
    Then it is easily checked that $Y$ is an ESSF again, with index of self-similarity $\alpha/\gamma$.
    Therefore, if $\alpha\neq 0$ and $\beta\neq 0$, taking $\gamma = \alpha/\beta$, one can transform any $\alpha$-ESSF into a $\beta$-ESSF, but note that one cannot get a homogeneous process with this transformation.
    As a result, there are really two classes of ESSF processes to consider: the $\alpha$-ESSF with $\alpha\neq 0$, which are a simple transformation away from being $1$-ESSF processes, and the so-called homogeneous $0$-ESSF processes.
    
    \item Note that this definition extends the classical case of Definition \ref{def:ssfrag-usual}.
    Indeed, if $\Pi$ is a usual $\alpha$-self-similar fragmentation process started from $\mathbf{1}$, then by definition, almost surely for all $t\geq 0$ and $i\in \N$, $\Pi(t)$ has asymptotic frequencies and one can define $V_i(t) := \abs{B}$ if $B$ is the block containing $i$ in $\Pi(t)$.
    Now consider an independent sequence $X^{(\cdot)}$ of copies of $(\Pi,\mathbf{V})$, and define for any  $x\in\markpart{\infty}$,
    \[
    X_x = \ssfrag_\alpha\!\big(x,X^{(\cdot)}\big).
    \]
    Then $X_x$ is the distribution of an $\alpha$-ESSF started from $x$, which extends the usual self-similar fragmentation $\Pi$ -- consider $x=(\mathbf{1},1)$ to obtain the original process.
    Note also that in this case $X$ is non-degenerate, because finite blocks have asymptotic frequency equal to $0$.
  \end{enumerate}
\end{remark}

As a first result about ESSF processes, let us show a projective Markov property.
It is very analogous to \cite[Lemma 3.2]{BDLS18} and \cite[Proposition 2]{Duc18}, but we need another statement in the present context.\begin{lemma} \label{lem:projmarkov}
  Let $X$ be an ESSF process.
  Then for any $n\in\N$, the process $(X(t)\restr{n},t\geq 0)$ is Markovian in $\markpart{n}$.
  More precisely, there exists a measurable transition kernel $(p^n_t,t\geq 0)$  on $\markpart{n}$ 
  such that for any initial state $x\in\markpart{\infty}$,
  \[
  \P_{x}(X(t)\restr{n}\in \cdot) = p^{n}_t(x\restr{n},\cdot\,).
  \]
\end{lemma}
\begin{proof}
  To show that $p^n$ is well-defined by the formula above, we need to prove that
  \[
  \P_{x}\big(X(t)\restr{n}\in \cdot\big) = \P_{x'}\big(X(t)\restr{n}\in \cdot\big)
  \]
  for any two initial states $x,x'\in\markpart{\infty}$ such that ${x'}\!\!\restr{n}=x\restr{n}$.
  
  Consider a probability space $(\Omega,\P)$ such that $X^{(\cdot)}$ is a sequence of i.i.d.\ copies of the ESSF process started from $(\mathbf{1},1)$, and let $\alpha\in\R$ be the self-similarity index of $X$.
  By the branching property, we have
  \begin{align*}
  & \P_{x}\big(X(t)\restr{n}\in \cdot\big) = \P\big( \!\ssfrag_\alpha(x,X^{(\cdot)})(t) \restr{n}\in \cdot\big)\\
  \text{and } & \P_{x'}\big(X(t)\restr{n}\in \cdot\big) = \P\big( \!\ssfrag_\alpha(x',X^{(\cdot)})(t) \restr{n}\in \cdot\big).
  \end{align*}
  It remains to notice that by definition, $\ssfrag_\alpha(x,X^{(\cdot)})(t) \restr{n}$ is in fact a functional which depends only on $x\restr{n}$ and $X^{(\cdot)}$.
  Therefore, because ${x'}\!\!\restr{n}=x\restr{n}$, we have
  \[
  \ssfrag_\alpha(x,X^{(\cdot)})(t) \restr{n} = \ssfrag_\alpha(x',X^{(\cdot)})(t) \restr{n}
  \]
  everywhere on $\Omega$, which implies by the preceding display that
  \[
  \P_{x}\big(X(t)\restr{n}\in \cdot\big) = \P_{x'}\big(X(t)\restr{n}\in \cdot\big).
  \]
  This shows that $p^n$ is well-defined.
  It remains to note that this is indeed a measurable transition kernel, i.e.\ to show that for any Borel set $A\subset\markpart{n}$, the map $(x,t)\in\markpart{n}\times[0,\infty)\mapsto p^n_t(x,A)$ is measurable.
  This is the case because the map $(x,t)\in\markpart{\infty}\times[0,\infty)\mapsto \P_{x}(X(t)\in A)$ is measurable by definition of $X$ as a strong Markov process, and we can write
  \[
  p^n_t(x,A) = \P_{\phi_n(x)}(X(t)\in A)
  \]
  for any well-chosen measurable map $\phi_n:\markpart{n}\to\markpart{\infty}$ such that $\phi_n(x)\restr{n}=x$, for instance such that $\phi_n(x)$ is the unique such marked partition with a unique infinite block $\{n+1,n+2,\ldots\}$ with mark $0$.
\end{proof}

The previous lemma shows that given an ESSF process $X$, one can define its law started from any $x_0\in\markpart{n}$, for any $n\in \N$, as the law of the restriction $X\restr{n}$ of the initial process started from any $x\in\markpart{\infty}$ such that $x\restr{n}=x_0$.

As a result, the restriction $X\restr{1}$ of an ESSF process $X=(\Pi,\mathbf{V})$ to $\markpart{1}=\partit{1}\times S^1$ is a Markov process.
Since the space $\partit{1}$ is a singleton, the lemma implies that the real-valued process $V_1=(V_1(t),t\geq 0)$ is a Markov process in $S^1$ and note that by exchangeability, the process $V_i$ has the same marginal distribution for all $i\geq 1$.
Further, Definition \ref{def:ESSF} implies that it is an a.s.\ càdlàg strong Markov process satisfying a self-similar property; more precisely, for $v\geq 0$ let $P_v$ denote the distribution of $V_1$ started at $v$ on the Skorokhod space of càdlàg maps $[0,\infty)\to S^1$, and let $V$ denote the canonical process on that space.
Then
\[
(V(t),\,t\geq 0) \text{ under } P_v \ed (vV(v^\alpha t),\,t\geq 0) \text{ under }P_1,
\]
where $\alpha$ is the self-similarity index of $X$.
In other words, $V$ is a positive self-similar Markov process (pssMp).
Note that in the literature, the index of self-similarity of a pssMp refers in general to $-\alpha$ \cite{PR13} or $-1/\alpha$ when $\alpha\neq 0$, e.g.\ in~
\cite{Lam72} where Lamperti calls this the \emph{order} of the process rather than the index.
Here we use the convention found in the self-similar fragmentation literature, e.g.~\cite{Ber02,Ber02b,Kre09}.
Let us summarize in a proposition some properties of $V$ that can be deduced from the well-developed theory of self-similar Markov processes.
First, if $X=(\Pi,\mathbf{V})$ is an ESSF process, for each $i\geq 1$ define $\zeta_i := \inf\{t\geq 0,\, V_i(t)=0\}$, and for $t\in[0,\zeta_i]$,
\[
\phi_i(t) := \int_{0}^{t}V_i(t)^{\alpha}\,\d s.
\]
Note that $\phi_i$ is continuous and increasing.
We define its right-continuous inverse $\tau_i(t)$, for $t\in[0,\infty)$, by
\[
\tau_i(t) := \begin{cases}
\phi_i^{-1}(t) &\text{if }t < \phi_i(\zeta_i),\\
\infty &\text{if }t\geq \phi_i(\zeta_i).
\end{cases}
\]
We need a convention for infinite times, so we let $V_i(\infty)\equiv 0$, so that $V_i(\tau_i(t))$ is always defined.
Also, note that the definition of the $\frag$ operator implies that a.s., $\Pi$ has nonincreasing sample paths for the \emph{finer-than} partial order -- $\pi$ is finer than $\pi'$ if the blocks of $\pi'$ can be written as unions of blocks of $\pi$.
Since a.s.\ for all $n\geq 1$, $\Pi(t)\restr{n}$ is nonincreasing in a finite set, it is eventually constant.
This implies that $\Pi(t)$ converges a.s.\ when $t\to\infty$, and we may denote its limit by $\Pi(\infty)$.
Let us now state the proposition.

\begin{prop}  Let $X=(\Pi,\mathbf{V})$ be an $\alpha$-ESSF process and $i\geq 1$, and define
  \[
  \begin{gathered}
  \zeta_i := \inf\{t\geq 0,\, V_i(t)=0\},\\
  \phi_i(t) := \int_{0}^{t}V_i(t)^{\alpha}\,\d s, \qquad t\in [0,\zeta_i] \\
  \tau_i(t) := \phi_i^{-1}(t),\qquad t\geq 0
  \end{gathered}
  \]
  Then the following properties hold.
  \begin{itemize}
    \item Either $\zeta_i<\infty$ $\P_{(\mathbf{1},1)}$-a.s., or $\zeta_i=\infty$ $\P_{(\mathbf{1},1)}$-a.s.
    \item Either $\phi_i(\zeta_i)<\infty$ $\P_{(\mathbf{1},1)}$-a.s., or $\phi_i(\zeta_i)=\infty$ $\P_{(\mathbf{1},1)}$-a.s.
    \item In the case $\zeta_i < \infty$, either $V_i$ reaches $0$ continuously $\P_{(\mathbf{1},1)}$-a.s., or $V_i$ eventually jumps to $0$ $\P_{(\mathbf{1},1)}$-a.s.
    \item $\phi_i(\zeta_i) = \infty$ iff $V_i$ reaches $0$ continuously.
    \item The process $\xi_i := \log(V_i\circ\tau_i)$, -- i.e.\ defined by
    \[
    \xi_i(t) = \log V_i(\tau_i(t)), \qquad 0\leq t< \phi_i(\zeta_i),
    \]
    is a (killed in the case $\phi_i(\zeta_i) < \infty$) Lévy process called the inverse \emph{Lamperti transform} of $V_i$.
  \end{itemize}
\end{prop}
\begin{proof}
  These are classical results on pssMp, we refer to \cite{Lam72} for a proof.
\end{proof}

This proposition tells us that it is natural to consider the time-changed processes $V_i\circ \tau_i$ for $n\geq 1$, which behave as exponentials of Lévy processes.
However, there is no unique time-change that could make the whole process $X$ behave nicely.
Instead, we have to rely on stopping lines, which are tools generalizing stopping times in the context of branching Markov processes (see e.g.\ \cite{Cha91} for their use in branching Brownian motion, or \cite{Ber02,Ber06} in the context of fragmentations).

\subsection{Stopping lines, changing the index of self-similarity}

First let us define some filtrations associated with an ESSF process $X=(\Pi,\mathbf{V})$.
To this aim, let us endow the power set $2^{\N}:=\{A\subset\N\}$ with the topology generated by the metric $d(A,B) := \big(\!\sup\{n\in\N, A\cap[n] =B\cap[n]\}\big)^{-1}$, which makes $2^{\N}$ a compact space.
Now for $i\in \N$, let us define the block process $(B_i(t),t\geq 0)$ as the $2^{\N}$-valued càdlàg process such that for all $t\geq 0$, $B_i(t)$ is the block of $\Pi(t)$ containing $i$, that is:
\[
B_i(t) = \{j\in\N, \,i\sim^{\Pi(t)}j \}.
\]
Now we can define a sequence of natural filtrations associated with $X$ by
\[
\mathcal{G}_i=(\mathcal{G}_i(t),\,t\geq 0) \quad \text{ with }\mathcal{G}_i(t) = \sigma\big(B_i(s),V_i(s),\; s\in[0,t]\big), \qquad i\geq 1, t\geq 0.
\]

\begin{definition}
  Let $X=(\Pi,\mathbf{V})$ be an ESSF process.
  A sequence $L=(L_i,i\geq 1)$ of random variables with values in $[0,\infty]$ is called a stopping line if
  \begin{enumerate}[(i)]
    \item for all $i\geq 1$, $L_i$ is a $\mathcal{G}_i$-stopping time.
    \item for $i,j\geq 1$, if $i\sim^{\Pi(L_i)} j$, then $L_i=L_j$.
  \end{enumerate}
  Since (ii) entails that $i\sim^{\Pi(L_i)} j$ is an equivalence relation, its equivalence classes form a well-defined partition of $\N$ which we denote by $\Pi(L)$ with a slight abuse of notation.
  Also, denoting $\mathbf{V}(L)$ as the vector $(V_i(L_i),i\geq 1)$, it is clear that $X(L):=(\Pi(L),\mathbf{V}(L))$ is a well-defined (random) element of $\markpart{\infty}$.
\end{definition}

\begin{remark} \label{rk:stopping_lines}
  A fixed time $t\geq 0$ can be seen as a stopping line (an $L$ for which $L_i\equiv t$ for all $i\geq 1$), and it is easily checked that for a stopping line $L$, one can define $L+t$ and $L\wedge t$ by
  \[
  (L+t)_i = L_i+t \qquad \text{ and }(L\wedge t)_i = L_i \wedge t,
  \]
  which are again stopping lines.
  Thus for a stopping line $L$ we will be able to consider the processes $X(L+\cdot) := (X(L+t),\,t\geq 0)$ and $X(L\wedge\cdot) := (X(L\wedge t),\,t\geq 0)$.
  Since it will be useful, we define the following $\sigma$-algebra:
  \[
  \mathcal{G}_L := \sigma\big(X(L\wedge t),\,t\geq 0\big).
  \]
\end{remark}

We can now state the Markov property for stopping lines, which is analogous to what can be found in \cite[Lemma 3.14]{Ber06}.

\begin{prop}[Stopping line Markov property] \label{prop:stopline}
  Let $X$ be an $\alpha$-ESSF, and $L$ be a stopping line.
  Then conditional on $\mathcal{G}_L$, the following equality in distribution holds:
  \begin{equation}\label{eqprop:markov_lines}
  X(L+\cdot) \ed \ssfrag_\alpha\!\big(X(L),X^{(\cdot)}\big),
  \end{equation}
  where $X^{(\cdot)}$  is an independent, i.i.d.\ sequence of copies of the process started from $(\mathbf{1},1)$.
\end{prop}

\begin{proof}
  See Appendix \ref{sec:proof-stopline}.
\end{proof}

The next step in the analysis of ESSF processes is to bring the index of self-similarity to $0$.
This will be done via the random time changes $(\tau_i(t), i\geq 1, t\geq 0)$ defined above by
\[
\tau_i(t) := \phi_i^{-1}(t), \quad \text{ where } \quad \phi_i(u) := \int_{0}^{u}V_i(s)^{\alpha} \,\d s, \quad u\geq 0.
\]
These time changes enable us to turn an $\alpha$-ESSF into a homogeneous ESSF.
The following proposition makes this claim more precise.

\begin{prop} \label{prop:changing-index}
  Let $X=(\Pi,\mathbf{V})$ be an $\alpha$-ESSF, with $\alpha\in\R$, and let $\beta\in\R$.
  For $i\geq 1$, define the random map $\tau^{\beta}_i$ as the right-continuous inverse of the map $u\mapsto \int_{0}^{u}V_i(s)^{\beta} \,\d s$, that is:
  \[
  \forall t\geq 0,\qquad\tau^{\beta}_i(t) = \left (\int_{0}^{\cdot}V_i(s)^{\beta} \,\d s\right )^{-1}(t) \;=\;\inf \bigg\{u\geq 0,\,\int_{0}^{u}V_i(s)^{\beta} \,\d s=t\bigg\},
  \]
  with the convention $\inf \emptyset=\infty$.
  Then for each $t\geq 0$, $\tau^{\beta}(t)$ is a stopping line, and the process $X\circ\tau^{\beta} := (X(\tau^{\beta}(t)),\,t\geq 0)$ is an $(\alpha-\beta)$-ESSF.
  Furthermore, if $X$ is non-degenerate, then $X\circ\tau^{\beta}$ is also non-degenerate.
\end{prop}

\begin{proof}
  See Appendix \ref{sec:proof-changing-index}.
\end{proof}

By bringing the index of self-similarity to $0$ we can transform any ESSF into a homogeneous process.
It follows from the definition that if $X=(\Pi, \mathbf{V})$ is a $0$-ESSF, then $\Pi$ is simply a classical homogeneous fragmentation process (possibly stopped at a random line if the marks can jump to $0$), and the process of marks $\mathbf{V}$ can be seen as a secondary process, evolving jointly and branching simultaneously with $\Pi$.
As in the classical theory, studying the homogeneous case is necessary to describe and give results on the general self-similar case.
We study further those $0$-ESSF in Section \ref{sec:main-res}, but first let us point out how self-similar growth-fragmentations can be understood as special cases of ESSF processes.

\paragraph{Link with self-similar growth-fragmentation processes}

Consider an ESSF $X=(\Pi,\mathbf{V})$ and define the associated point-measure valued process $\tilde{\mathbf{V}}$ by
\[
\tilde{\mathbf{V}}(t) := \sum_{k\geq 1}\delta_{\tilde{V}_k(t)},
\]
where $\tilde{V}_k(t)$ denotes the mark associated with the $k$-th block of $X(t)$.
Assume that the fragmentation events are \emph{binary} and \emph{conservative} in the sense that for all $t\geq 0$ such that $\Pi(t-) \neq \Pi(t)$,
\begin{itemize}
  \item there is a unique block $B_k$ of $\Pi(t-)$ such that $\Pi(t)_{|B_k} \neq \Pi(t-)_{|B_k} = \mathbf{1}_{B_k}$, and this block can be written as the union of two blocks of $\Pi(t)$, say $B_{k_1}$ and $B_{k_2}$.
  \item $\tilde{V}_k(t-)=\tilde{V}_{k_1}(t)+\tilde{V}_{k_2}(t)$ holds, where $k,k_1,k_2$ are as above.
\end{itemize}
Then it follows from the definition of ESSF processes that $\tilde{\mathbf{V}}$ is a self-similar growth-fragmentation process as introduced by Bertoin \cite{Ber17}.
In fact, any self-similar growth-fragmentation process can be built as the process of marks of a well-chosen ESSF.
We will make this claim precise after giving a characterization of ESSF processes, in Section \ref{sec:link-with-gf}.


\section{Main results} \label{sec:main-res}

\subsection{Decomposition of ESSF processes}

Let us consider here a homogeneous $0$-ESSF process $X=(\Pi,\mathbf{V})$, started from $(\mathbf{1},1)$.
We know by Lemma \ref{lem:projmarkov} that it satisfies a projective Markov property, i.e.\ for all $n\in\N$, $X\restr{n}$ defines a Markov process with values in $\partit{n}$.
Let $n\in\N$ be fixed, and define the stopping time 
\[
T_n := \inf\{t\geq 0,\, \Pi(t)\restr{n} \neq \mathbf{1}_n \text{ or } V_1(t) = 0\},
\]
as well as the killed process
\[
\tilde{\xi}_n := (\log V_1(t), \,0\leq t < T_n).
\]
Note that homogeneity implies that the pair $(\tilde{\xi}_n-\log v,T_n)$ has the same distribution under every $\P_{(\mathbf{1},v)}$ for all $v\in S^1\setminus\{0\}$.
Therefore for $t\geq 0$, conditional on $\{T_n>t\}$, the Markov property applied at time $t$ shows that $(\tilde{\xi}_n(t+\cdot)-\tilde{\xi}_n(t),T_n-t)$ has the same distribution as $(\tilde{\xi}_n,T_n)$ under $\P_{(\mathbf{1},1)}$.
This shows that the killed process $\tilde{\xi}_n$ is distributed as
\[
\tilde{\xi}_n\ed(\xi_n(t),\,0\leq t < T'_n),
\]
where $\xi_n$ is a Lévy process and $T'_n$ is an independent exponential random variable.
Up to enriching our probability space, whenever we need to, we will assume that the whole (non-killed) Lévy process $\xi_n=(\xi_n(t),t\geq 0)$ is defined in our space and that the above equality in distribution holds almost surely, that is,
\[
\forall n\geq 1, \qquad \tilde{\xi}_n = (\log V_1(t),\,0\leq t < T_n) = (\xi_n(t),\,0\leq t < T_n) \qquad \text{a.s.}
\]
Note that this implies that if $T_n<\infty$, then $V_1(T_n-) = \exp(\xi_n(T_n))> 0$.
Now for $n\in \N$ such that $T_n<\infty$ almost surely, consider $D_n$, the dislocation (or freezing) at time $T_n$, defined by
\[
D_n := (\Pi(T_n),\mathbf{V}(T_n)/V_1(T_n-))\restr{n} \in\markpart{n},
\]
where the division $\mathbf{V}(T_n)/V_1(T_n-)$ is to be understood coordinate-wise.
Equivalently, $D_n$ is the unique random marked partition such that
\[
X(T_n)\restr{n} = \frag\!\big(X(T_n-)\restr{n},D_n\big),
\]
with a slight abuse of notation in this case since $X(T_n-)\restr{n}$ has only one block ($D_n$ is not a sequence but additional terms are useless to define a fragmentation of a single block).

Note that this implies that $D_n$ has the same distribution under every $\P_{(\mathbf{1},v)}$ for all $v\in S^1\setminus\{0\}$.
Thus for any bounded measurable maps $g:\R\to\R$, $h:\markpart{n}\to\R$ and $t\geq 0$, applying the Markov property at time $t\geq 0$, one gets
\[
\E_{(\mathbf{1},1)}[g(\xi_n(t))\1_{T_n>t} h(D_n)] = \E_{(\mathbf{1},1)}[g(\xi_n(t))\1_{T_n>t}] \, \E_{(\mathbf{1},1)}[h(D_n)],
\]
which shows that the killed Lévy process $(\xi_n,T_n)$ and the marked partition $D_n$ are independent.
Let us define $\mathcal{D}_n$ as the law of $D_n$, and notice also that exchangeability of $X\restr{n}$ implies that $\mathcal{D}_n$ is an exchangeable probability measure on $\markpart{n}$.

Since $(\xi_n,T_n)$ is a killed Lévy process, one can define uniquely $d_n\in\R$, $\beta_n\geq 0$, $J_n\geq 0$ and $\lambda_n$ a measure on $\R\setminus\{0\}$ satisfying $\int 1\wedge y^2 \, \lambda_n(\d y) < \infty$, such that
\begin{itemize}
  \item the process $\xi_n$ is a Lévy process with characteristic exponent
  \[
  \psi_n(\theta) := \log\E[\e^{i\theta\xi_n(1)}] = id_n\theta -\frac{\beta_n}{2}\theta^2 + \int_{\R}\left (\e^{i\theta y} - 1 - i\theta y \1_{\abs{y}\leq 1}\right )\,\lambda_n(\d y),
  \]
  \item $\xi_n$ is killed at rate $J_n = 1/\E[T_n]$, which may be $0$ if $T_n = \infty$ almost surely.
\end{itemize}

\begin{remark}\label{rk:construction-from-characs}
  Note that knowing $(\psi_n, J_n, \mathcal{D}_n)$ for $n\in\N$ is enough to reconstruct the process $X$.
  Indeed, starting from $(\mathbf{1},v)$, the process $X\restr{n}$ up to time $T_n$ has distribution equal to that of
  \[
  Y_n := \big((\mathbf{1}_n, v\e^{\xi_n(t)}),\, 0\leq t < T_n\big),
  \]
  and at time $T_n$ jumps to $(\Pi, v\e^{\xi_n(T_n-)}\mathbf{V})$, where $(\Pi,\mathbf{V})$ is independently drawn according to $\mathcal{D}_n$.
  
  By the branching property, one only needs to iterate this construction at each jump time, independently for each marked block, to get the whole process $X\restr{n}$.
  By Kolmogorov's extension theorem -- since $(X\restr{m})\restr{n}=X\restr{n}$ for each $n\leq m$ -- these distributions characterize the distribution of $X$.
\end{remark}

Let us now state our main result, which identifies the form that those characteristics can take.

\begin{theorem} \label{thm:main}
  Let $X$ be a non-degenerate $0$-ESSF and for each $n$, write $(\psi_n, J_n, \mathcal{D}_n)$ for the characteristics describing the law of $X\restr{n}$.
  Then there is a unique quadruple $(c,d,\beta,\Lambda)$, where $c,\beta\geq 0$, $d\in\R$, and $\Lambda$ is a measure on $\markmasspart\setminus\{(\mathbf{1},1)\}$, which satisfies necessarily
  \begin{equation} \label{eq:lambda-integrability}
  \int_{\markmasspart}\big(1-s_1\1_{v_1>0} + (\log v_1)^2 \wedge 1\big)\,\Lambda(\d \mathbf{z}) < \infty,
  \end{equation}
  such that for all $n\in\N$,
  \begin{enumerate}[(i)]
    \item $\displaystyle\psi_n(\theta) = id\theta -\frac{\beta}{2}\theta^2 + \int_{\markmasspart}\bigg(\sum_{\substack{j\geq 1\\ v_j>0}}(s_j)^n\big(\e^{i\theta\log v_j} - 1\big) - i\theta\log v_1 \1_{\abs{\log v_1}\leq 1}\bigg)\,\Lambda(\d \mathbf{z}).$
    \item $J_n = nc + \displaystyle \int_{\markmasspart}\Big (1-\sum_{\substack{i\geq 1\\ v_i>0}}(s_i)^n\Big )\,\Lambda(\d \mathbf{z})$.
    \item if $J_n>0$, $\displaystyle \mathcal{D}_n=\frac{1}{J_n}\Big(\sum_{i=1}^{n}c\delta_{\mathfrak{e}^{n}_i} + \int_{\markmasspart}\rho^{n}_{\mathbf{z}}\big(\cdot\cap\{\pi\neq\mathbf{1}_n\text{ or }(\pi,\mathbf{v})=(\mathbf{1}_n,0)\}\big)\,\Lambda(\d\mathbf{z})\Big)$, where $\rho^{n}_{\mathbf{z}}$ is the paintbox process defined in Section \ref{sec:paintbox}, and $\delta_{\mathfrak{e}^{n}_i}$ denotes the Dirac point measure on $\mathfrak{e}^{n}_i$, the marked partition defined as
    \[
    \mathfrak{e}^{n}_i := \Big(\big\{[n]\setminus\{i\},\{i\}\big\},\, (1,\ldots,1,\underbrace{0}_{i-\text{th index}},1,\ldots,1) \Big).
    \]
  \end{enumerate}
  Conversely if $c,\beta \geq 0$, $d\in\R$ and $\Lambda$ is a measure on $\markmasspart\setminus\{(\mathbf{1},1)\}$, satisfying \eqref{eq:lambda-integrability}, then there exists a $0$-ESSF with characteristics as above.
\end{theorem}

\begin{proof}
  See Appendix \ref{sec:proof-main}.
\end{proof}

\begin{remark}~
  \begin{itemize}
    \item 
    While this is not immediate, integrability condition \eqref{eq:lambda-integrability} is enough to show that the integrals appearing in the statement of the theorem are well-defined.
    This is explicit in the last part of the proof -- for the converse statement.
    
    \item 
    It is an immediate consequence of the theorem that the process describing the block of $X$ containing $1$ can be constructed in the following Poissonian way.
    Consider $\mathcal{N}$ a Poisson point process on $[0,\infty)\times\markpartstar$ with intensity
    \[
    \d t\otimes\Big(\sum_{i=1}^{\infty}c\delta_{\mathfrak{e}^{n}_i} + \int_{\markmasspart}\rho^{n}_{\mathbf{z}}\big(\cdot\big)\,\Lambda(\d\mathbf{z})\Big),
    \]
    and define
    \[
    \mathcal{N}' := \big\{(t,\log v_1), \; (t,x)\in\mathcal{N} \text{ with }x =(\pi,\mathbf{v}) \text{ and } v_1\notin \{0,1\}\big\},
    \]
    which has intensity $\d t \otimes\lambda_1$, where $\lambda_1$ is defined by
    \[
    \int_{\R} f\,\d\lambda_1 = \int_{\markmasspart}\sum_{\substack{j\geq 1\\ v_j\notin\{0,1\}}}s_j \,f(\log v_j)\,\Lambda(\d\mathbf{z}),
    \]
    and is the Lévy measure of the process $\xi_1$.
    It is clear that one can build a Lévy process $(\xi_1(t),t\geq 0)$ having characteristic exponent $\psi_1$ given by \textit{(i)} in the theorem and whose point process of jumps is exactly $\mathcal{N}'$.
    
    Define $(B(t),t\geq 0)$ as the $2^{\N}$-valued process given by
    \[
    B(t) = \bigcap_{\substack{0\leq s <t\\(s,x)\in\mathcal{N}}}A(x),
    \]
    where $A(x)\subset \N$ denotes the block of $x$ containing $1$.
    Also, for any $n\in\N$, define
    \[
    \widetilde{T}_n := \inf \big\{t\geq 0, \; (t,x)\in\mathcal{N}\text{ with }x=(\pi,\mathbf{v})\text{ such that }\pi\restr{n}\neq \mathbf{1}\text{ or }v_1=0\big\} \sim \Exp(J_n).
    \]
    Now $(B(t),\e^{\xi_1(t)},0\leq t < \widetilde{T}_1)$ is distributed as the marked block containing $1$ in $X$ and by construction, we also get the following equality in distribution
    \[
    \Big(X(T_n)\restr{n},\,n\in\N\Big) \ed \Big(\big(\pi_n, \,\e^{\xi_1(\widetilde{T}_n-)}\mathbf{v}_n\big)\restr{n},\,n\in\N\Big) ,
    \]
    where $x_n=(\pi_n,\mathbf{v}_n)$ is the element of $\markpartstar$ such that $(\widetilde{T}_n,x_n)\in \mathcal{N}$.
  \end{itemize}
\end{remark}

\medskip

Combining Theorem \ref{thm:main} with Proposition \ref{prop:changing-index}, we get the following characterization of all ESSF processes.

\begin{corollary}
  Let $X$ be a non-degenerate $\alpha$-ESSF.
  Then there exists a unique quadruple $(c,d,\beta,\Lambda)$ as in Theorem~\ref{thm:main} such that if $(\tau^{\alpha}(t),t\geq 0)$ are the stopping lines as defined in Proposition~\ref{prop:changing-index}, then $X\circ\tau^{\alpha}$ is a homogeneous ESSF with characteristics $(c,d,\beta,\Lambda)$.
\end{corollary}

Let us point out that condition \eqref{eq:lambda-integrability} is surprisingly nonrestrictive.
There are no integrability assumptions concerning the marks of the smallest blocks (with labels greater than $1$).
Consequently, the point measure $\sum_k\delta_{\tilde{V}_k(t)}$, where $\tilde{V}_k(t)$ denotes the mark of the $k$-th block in $X(t)$, might assign infinite mass to any interval $(a,b)\subset [0,\infty)$ for any $t>0$.
Indeed it suffices for instance that $\Lambda(\d \mathbf{z})$ be of the form
\[
\int_{\markmasspart}\prod_{i\geq 1}F_i(s_i,v_i)\,\Lambda(\d \mathbf{z}) = \int_{(0,1)\times S^{1}} F_1(s_1,v_1)\,\E\bigg[\prod_{i\geq 2}F_i\Big(\frac{1-s_1}{2^{i-1}},Z_i\Big)\bigg]\,\nu(\d z_1),
\]
where $\nu$ is a measure on $(0,1)\times S^{1}$ with infinite mass and satisfying
\[
\int_{(0,1)\times S^{1}}\big(1-s_1\1_{v_1>0} + (\log v_1)^2 \wedge 1\big)\,\nu(\d z_1) < \infty,
\]
and $Z_2,Z_3,\ldots$ are i.i.d.\ $\Exp(1)$ random variables.
On the other hand, if one assumes an integrability condition such as
\[
\int_{\markmasspart}\big(\sum_{i\geq 1} v_i^{\theta}\big) - 1 -\theta\log v_1 \1_{\abs{\log v_1}\leq 1}\,\Lambda(\d \mathbf{z}) < \infty
\]
for some $\theta\in\R$, then one observes a process of point measures $(\sum_k\delta_{\tilde{V}_k(t)},t\geq 0)$ that is nice in the sense that for all $t\geq 0$, $\E\sum_k\tilde{V}_k(t)^{\theta}<\infty$.
This is the object of the next section.

\subsection{Absorption in finite time}
Consider here a non-degenerate $\alpha$-ESSF with characteristics $(c,d,\beta,\Lambda)$, started from $(\mathbf{1},1)$.
We are interested in the case where the pssMp $V_1$ reaches $0$ in finite time and, if $T_i$ denotes the hitting time of $0$ by the process $V_i$, we aim at giving a sufficient condition for which
\[
\sup_{i\in\N} T_i < \infty \qquad \text{a.s.}
\]
If this holds, then at this random time the process is frozen in a state $X(\infty)=(\Pi(\infty),\mathbf{V}(\infty))$ where $\mathbf{V}(\infty)=\mathbf{0}=(0,0,\ldots)$, and we say the process $X$ is \emph{absorbed in finite time}.
We first put aside a trivial case and assume that
\[
c>0 \quad \text{ or }\quad \int_{\markmasspart}(1-s_1)\,\Lambda(\d \mathbf{z}) > 0,
\]
since otherwise $\Pi$ would be almost surely constant equal to the coarsest partition $\{\N\}$.
Recall that a classical self-similar fragmentation which is absorbed in finite time (this is always true when $\alpha < 0$ \cite[Proposition 2]{Ber03}) is always totally fragmented in the limit, in the sense that $\Pi(\infty)$ is the partition of $\N$ into singletons.
Clearly in our case, because of the possible freezing of blocks at dislocation events, $\Pi(\infty)$ is almost surely totally fragmented if and only if
\[
\forall i\geq 1, \,s_i>0\implies v_i>0 \qquad \Lambda\text{-a.e.\ on }\markmasspart.
\]
A stronger property than absorption in finite time is the following: we say $X$ has \emph{finite total length} if
\begin{equation*}  \int_{0}^{\infty} \# X(t) \, \d t <\infty \quad \text{a.s.},
\end{equation*}
where $\#x$ denotes the number of blocks with positive mark in the marked partition $x$.
One can interpret this quantity as the total length of the tree describing the genealogy of blocks in the fragmentation, hence the name.
Note that this implies that for a fixed time $t\geq 0$, $\# X(t)$ is almost surely finite, which is well-known \cite[Proposition 2]{Ber03} in the classical self-similar fragmentation case for $\alpha < -1$.

In this section our aim is to provide sufficient conditions for ESSF processes to be absorbed in finite time and to have finite total length.
The following result extends the classical setting, and makes use of natural martingales -- the so-called additive martingales -- appearing in the homogeneous case.
In order to be able to state it, we need a couple of additional definitions.
For a marked partition $x=(\pi,\mathbf{v})\in\markpart{n}$ with $n\in\N\cup\{\infty\}$, and $\theta\in\R$, let us write
\begin{equation} \label{eqdef:stheta}
S_\theta(x) := \sum_{k\in\N}\tilde{v}_k^{\theta},
\end{equation}
where $\tilde{v}_k$ denotes the mark associated with the $k$-th block of $x$.
Let us also introduce $\kappa : \R\to(-\infty,\infty]$ defined by
\begin{equation} \label{eqdef:kappatheta}
\kappa(\theta) := d\theta + \frac{\beta}{2}\theta^2 + \int_{\markmasspart}\Big(\big(\sum_{i\geq 1} v_i^{\theta}\big) -1 -\theta \log v_1 \1_{\abs{\log v_1} \leq 1} \Big)\,\Lambda(\d \mathbf{z}).
\end{equation}
Note that the integral in the last display is well-defined with values in $(-\infty,\infty]$, since
\begin{align*}
1 +\theta \log v_1 \1_{\abs{\log v_1} \leq 1} - \sum_{i\geq 1} v_i^{\theta}  &\leq \big(1+\theta\log v_1\1_{\abs{\log v_1} \leq 1}-v_1^{\theta} \big)_{+}\\
&\leq C \big((\log v_1^2)\wedge 1\big)
\end{align*}
where $C$ is a positive constant which depends on $\theta$, so the negative part of the integrand in the definition of $\kappa$ is $\Lambda$-integrable.

\begin{prop} \label{prop:compactness}
  Let $X$ be a non-degenerate $0$-ESSF with characteristics $(c,d,\beta,\Lambda)$ started from $(\mathbf{1},1)$.
  For all $\theta\in\R$ and $t\geq 0$,
  \[
  \E[S_\theta(X(t))] = \e^{t\kappa(\theta)},
  \]
  with $S_\theta$ and $\kappa(\theta)$ respectively defined as in \eqref{eqdef:stheta} and \eqref{eqdef:kappatheta}.
  These quantities may be infinite,
  meaning that if $\kappa(\theta)=\infty$, then $\E[S_\theta(X(t))]=\infty$ for all $t>0$.
  If there is $\theta\in\R$ such that $\kappa(\theta) < \infty$, then the process
  \[
  \big(\e^{-t\kappa(\theta)}S_\theta(X(t)),\; t\geq 0\big)
  \]
  is a martingale.
  If there is $\theta\neq 0$ such that $\kappa(\theta) < 0$, then for any $\alpha\in\R$:
  \begin{itemize}
    \item if $-\alpha/\theta > 0$, the $\alpha$-ESSF with characteristics $(c,d,\beta,\Lambda)$ is absorbed in finite time.
    \item if $-\alpha/\theta \geq 1$, the $\alpha$-ESSF with characteristics $(c,d,\beta,\Lambda)$ has finite total length.
  \end{itemize}
\end{prop}

\begin{proof}
  See Appendix \ref{sec:proof-compactness}.
\end{proof}

\begin{remark}
  For a classical self-similar fragmentation with erosion coefficient $c\geq 0$ and dislocation measure $\nu$, we have
  \[
  \kappa(\theta) = -c\theta + \int_{\masspart}\Big(\sum_{i\geq 1} s_i^{\theta} - 1 \Big)\, \nu(\d \mathbf{s}).
  \]
  Since $\sum_i s_i \leq 1$ $\nu$-a.e., for all $\theta > 1$ we have $\kappa(\theta) < 0$, so we recover absorption in finite time for any $\alpha < 0$ and finite total length for any $\alpha < -1$.
\end{remark}

\begin{remark}
  Let us also mention that one can model branching Brownian motion in our setting.
  Indeed, consider a homogeneous ESSF where the logarithm of marks follow drifted Brownian motion and blocks dislocate at rate one into two blocks (say both with asymptotic frequency equal to half of the mother block) carrying the same mark.
  More precisely, take a {$0$-ESSF} with characteristics $c=0$, $d\in\R$, $\beta=1$ and with $\Lambda(\d \mathbf{z})$ the Dirac measure on $((\frac{1}{2},1),(\frac{1}{2},1),0,\ldots)$.
  
  Then the point process recording the positions of the logarithm of marks
  \[
  \sum_{k\in\N}\delta_{\log \tilde{V}_k(t)},\qquad t\geq 0
  \]
  is a classical binary branching Brownian motion with drift $d$.
  One gets a cumulant function
  \[
  \kappa(\theta) = d\theta +\frac{\theta^2}{2} + 1.
  \]
  This polynomial in $\theta$ takes negative values if and only if $d^2 - 2 > 0$, and we essentially recover the well-known fact that if $d>\sqrt{2}$, the lowest particle of a branching Brownian motion with drift $d$ goes to~$+\infty$.
\end{remark}


\subsection{Self-similar growth-fragmentation processes as ESSF.} \label{sec:link-with-gf}
Growth-fragmentation processes as introduced by Bertoin \cite{Ber17} can be thought of as modeling \emph{cell systems}, i.e.\ families of cells which grow (their size vary with time in a Markovian way) and reproduce independently of one another.
To be more specific, the size of a single cell is a Markov process $Y=(Y(t),t\geq 0)$ with values in $[0,\infty)$ with càdlàg paths and no positive jumps, called the \emph{cell process}.
At each negative jump of the process (a time $t$ such that $\abs{\Delta Y(t)} = Y(t-)-Y(t) > 0$), a new cell is born in the system with size $\abs{\Delta Y(t)}$.
The growth-fragmentation process $\mathbf{X} = (\mathbf{X}(t), t \geq 0)$ is the point-measure valued process recording the size of all cells alive at time $t$.

It is said to be self-similar if it can be built from a self-similar cell process $Y$, say with index $\alpha$ and characteristics $(d, \beta, \lambda, k)$, with $k,\beta \geq 0$, $d\in \R$ and $\lambda$ a Lévy measure on $(-\infty, 0)$.
This means that $Y$ is the Lamperti transform of a Lévy process $\xi$ satisfying $\E \e^{q\xi_1}=\exp(\phi(q))$ for all $q > 0$, with
\[
\phi(q) = -k + \frac{\beta}{2} q^2 + dq + \int_{(-\infty, 0)}\!\!\big(\e^{yq}-1-qy\1_{y>-1}\big)\,\lambda(\d q).
\]
Note that different cell processes may yield the same growth-fragmentation (see Shi \cite{Shi17}), and in fact the distribution of a growth-fragmentation is characterized uniquely by the index $\alpha$ and the cumulant function
\[
\kappa(q) := \phi(q)+\int_{(-\infty, 0)}\!\!(1-\e^{y})^q\,\lambda(\d q),
\]
which takes finite values for $q\geq 2$.
Furthermore, this map $\kappa$ has the following interpretation: if $\alpha=0$, then $\E \sum_{x\in\mathbf{X}(t)}x^q = \e^{t\kappa(q)}$.
A consequence of \cite{Shi17} is that $\kappa$ is unchanged (therefore the growth-fragmentation is the same) when replacing $\lambda$ by $\lambda_{|[-\log 2, 0)} + \bar{\lambda}$ (and modifying accordingly the drift term $d$), where $\bar{\lambda}$ is the measure on $[-\log 2, 0)$ defined as the push-forward of $\lambda_{|(-\infty, -\log 2)}$ by $y \mapsto \log(1-\e^{y})$.
Therefore we assume that $\lambda$ is a measure on $[-\log 2, 0)$, which intuitively means that the cell process $Y$ always keeps track of the biggest cell in a division, i.e.\ that $Y(t)\geq Y(t-)/2 \geq \abs{\Delta Y(t)}$ almost surely for all $t\geq 0$.

Let us now build a version of this growth-fragmentation process $\mathbf{X}$ as the process of marks in an ESSF.
Unfortunately there seems to be no natural candidate for this, since many ESSF may have the same process of marks.
Consider a map $Z:[-\log 2, 0)\to\markmasspart$ and write $Z(y)=(s_i(y), v_i(y), i\geq 1)$.
We assume that $v_1(y) = \e^{y}$, $v_2(y) = 1-\e^{y}$, $s_2(y) = 1-s_1(y)$, and $s_1$ is such that
\[
\int_{[-\log 2, 0)} 1 - s_1(y)\,\lambda(\d y)<\infty.
\]
For instance $s_1(y) = \e^{-y^2}$ or $s_1(y)=\e^{y}(1-y)$ both satisfy this condition, but of course there is no natural choice here.
Now define the push-forward $\Lambda:=\lambda\circ Z^{-1} + k\delta_{(\mathbf{1},0)}$, which satisfies by construction
\[
\int_{\markmasspart} \big(1-s_1\1_{v_1 > 0} +(\log v_1)^2 \wedge 1\big)  \,\Lambda(\d \mathbf{z}) < \infty,
\]
and consider the ESSF $X=(\Pi,\mathbf{V})$ with characteristics $(d, \beta, \Lambda)$ and $c=0$.
Then the process of marks $\tilde{\mathbf{V}} = (\sum_k\delta_{\tilde{V}_k(t)},\,t\geq 0)$ has the same distribution as the growth-fragmentation $\mathbf{X}$.
Indeed, up to a time-change, we can assume that $\alpha=0$, and Proposition \ref{prop:compactness} implies that for all $t\geq 0$, $\E \sum_k \tilde{V}_k(t)^q = \e^{t\kappa(q)}$, with
\begin{align*}
\kappa(q) &= dq +\frac{\beta}{2}q^2 + \int_{\markmasspart}\big(\sum_{i\geq 1} v_i^{q}\big) -1 -q \log v_1 \1_{\abs{\log v_1} \leq 1} \,\Lambda(\d \mathbf{z})\\
&= -k + dq +\frac{\beta}{2}q^2 + \int_{[-\log 2, 0)}\e^{yq} + (1-\e^{y})^{q} -1 -qy \,\lambda(\d y)\\
&= \phi(q) +  \int_{[-\log 2, 0)}(1-\e^{y})^{q}\,\lambda(\d y).
\end{align*}
This shows indeed that $\tilde{\mathbf{V}}\ed \mathbf{X}$.

\appendix
\section{Proofs} \label{sec:app}

\subsection{Proof of Proposition \ref{prop:paintbox}} \label{sec:proof-paintbox}

Let us write as usual $X=(\Pi,\mathbf{V})$.
First, note that $\Pi$ is an exchangeable partition with values in $\markpartstar$, therefore it has asymptotic frequencies -- so $\abs{X}^{\downarrow}$ exists almost surely -- and the finite blocks of $\Pi$ (if any) are necessarily singletons.
For the uniqueness part of the proposition, notice that any $\nu$ satisfying \eqref{eqprop:paintbox} must be equal to $\P(\abs{X}^{\downarrow}\in\cdot)$.

For the existence, let $(U_k,k\geq 1)$ be an i.i.d.\ sequence of uniform random variables on $[0,1]$, independent of $X$.
For every $i\in \N$, let $Z_i=(U_{k(\Pi,i)},V_i)\in[0,1]\times S^1$, where $k(\Pi,i)$ is the label of the block of $\Pi$ containing $i$
-- recall that labels are defined as an enumeration of blocks in increasing order of their smallest element: $k(\pi,1)=1$ for any partition $\pi$, then $k(\pi,i)=2$ for the first $i$ that is not in the same block as $1$, etc.

The key point of the proof is that the sequence $(Z_i,i\geq 1)$ is exchangeable, with values in a Polish space, which allows us to apply de Finetti's theorem.
To see that, consider a permutation $\sigma:\N\to\N$, and note that for each (deterministic) partition $\pi$, there exists a permutation $\phi:\N\to\N$ such that
\[
k(\pi,\sigma(i)) = \phi(k(\pi^{\sigma}, i)).
\]
This permutation of labels may not be unique (in the case where $\pi$ has a finite number of blocks), but we can assume that such a map $\phi$ is chosen deterministically as a function of $\pi$ and $\sigma$.
Going back to our random sequence $(Z_i,i\geq 1)$, we see that permuting its elements according to $\sigma$ yields for all $i\geq 1$,
\begin{align*}
Z_{\sigma(i)} &= (U_{k(\Pi,\sigma(i))}, V_{\sigma(i)})\\
&= (U_{\phi(k(\Pi^{\sigma},i))}, V_{\sigma(i)}).
\end{align*}
where $\phi$ is a random permutation whose randomness comes only from $\Pi$.
Now since the sequence $(U_k,k\geq 1)$ is i.i.d.\ and is independent from $\Pi$, if we define $U'_k=U_{\phi(k)}$ for each $k\geq 1$, then the sequence $(U'_k,k\geq 1)$ is still i.i.d.\ with the same distribution and independent from $\Pi$.
Finally, using the exchangeability property of $X=(\Pi,\mathbf{V})$, we have
\[
(Z_{\sigma(i)},i\geq 1) = \big((U'_{k(\Pi^{\sigma},i)},V_{\sigma(i)}),\;i\geq 1\big) \ed \big((U'_{k(\Pi,i)},V_{i}),\;i\geq 1\big)\ed \big((U_{k(\Pi,i)},V_{i}),\;i\geq 1\big),
\]
which shows that the sequence $(Z_i,i\geq 1)$ is exchangeable.

Therefore we can use Finetti's theorem: there exists a random probability measure $\theta \in \mathcal{M}_1([0,1]\times S^1)$ such that conditional on $\theta$, the sequence $(Z_i,i\geq 1)$ is i.i.d.\ with distribution $\theta$.
Let \[(a_k,k\geq 1)=(u_k,v_k,\,k\geq 1)\] denote the collection of atoms of $\theta$.
Note that $i\sim^{X}j$ iff $Z_i=Z_j$, and therefore the law of large numbers ensures us that the blocks of $X$ correspond to those atoms, i.e.\ for each $k\geq 1$, there is a block $B$ of $X$ with an asymptotic frequency $\abs{B}=\theta(a_k)$ and a mark equal to $v_k$.
Conversely any block which is not reduced to a singleton must be formed in this way.
Furthermore, note that singleton blocks have mark $0$ because of the assumption that $X\in\markpartstar$, so the knowledge of the atoms $(a_k,k\geq 1)$ and their mass is sufficient to reconstruct the sequence $(Z_i,i\geq 1)$, and therefore the marked partition $X$.

Now define for all $k\in\N$, $z_k := (\theta(a_k),v_k)$.
Up to a reordering, we can assume that $\mathbf{z}=(z_k,k\geq 1)$ is in $\markmasspart$ (if the sequence of atoms is finite, we concatenate to $\mathbf{z}$ infinitely many $(0,0)$ terms).
The previous discussion means that conditional on $\theta$, the asymptotic frequencies of $X$ are exactly
\[
\abs{X}^{\downarrow} = \mathbf{z} \in\markmasspart,
\]
and conditional on $\mathbf{z}$, the marked partition $X$ is drawn according to $\rho_{\mathbf{z}}$.
Note that the map $\theta\mapsto\mathbf{z}$ is measurable.
Indeed, by standard point processes arguments \cite[see e.g.][Lemma 9.1.XIII]{DV08}, there exists a measurable enumeration $(a_k,k\geq 1)$ of the atoms of $\theta$, and it is elementary that the nonincreasing reordering of this sequence is measurable.
Therefore, defining $\nu = \P(\abs{X}^{\downarrow}\in\cdot)$, which is the push-forward of the distribution of $\theta$ by the map $\theta\mapsto\mathbf{z}$, we see that it satisfies \eqref{eqprop:paintbox}.

\subsection{Proof of Proposition \ref{prop:stopline}} \label{sec:proof-stopline}

In this section, it will be helpful to consider the restriction of an ESSF process $X$ to a more general (and possibly random) subset $A\subset \N$, considered as a random variable living on the compact space $2^{\N}$.
We first consider a fixed -- non random -- $A\subset \N$ with cardinality $\#A\in\N\cup\{\infty\}$, and define a canonical enumeration of $A$ by
\[
\sigma_{A}:\begin{cases}
[\#A] &\to \N\\
i &\mapsto \min\{ n\in\N, \, \# (A\cap [n])=i \},
\end{cases}
\]
such that $A = \{\sigma_{A}(1),\sigma_{A}(2), \ldots\}$, with $\sigma_A(1) < \sigma_{A}(2) < \ldots$.
Now recall the definition of the action of injections on $\markpart{\infty}$.
For any $x\in\markpart{\infty}$, one can see $x^{\sigma_A}\in\markpart{\#A}$ as the restriction of $x$ to the set $A$.

As an inverse operation, for any $x'=(\pi',\mathbf{v}')\in\markpart{\#A},x''=(\pi'',\mathbf{v}'')\in\markpart{\#A^{\mathrm{c}}}$, where $A^{\mathrm{c}}:=\N\setminus A$, we can define
\[
x'\overset{A}{\oplus}x''\in\markpart{\infty}
\]
as the pair $(\pi,\mathbf{v})$ such that
\begin{gather*}
\forall i,j \in \N, \qquad i\sim^{\pi}j \iff \begin{cases}
&i,j\in A\text{ and }\sigma_{A}^{-1}(i) \sim ^{\pi'}\sigma_{A}^{-1}(j)\\
\text{or} & i,j\in A^{\mathrm{c}}\text{ and }\sigma_{A^{\mathrm{c}}}^{-1}(i) \sim ^{\pi''}\sigma_{A^{\mathrm{c}}}^{-1}(j)
\end{cases} \\
\text{and }{v}_{i} = \begin{cases}
{v}'_{\sigma_{A}^{-1}(i)} & \text{ if } i\in A\\
{v}''_{\sigma_{A^{\mathrm{c}}}^{-1}(i)} & \text{ if } i\in A^{\mathrm{c}}.
\end{cases}
\end{gather*}
Similarly, for processes $X'=(X'(t),\, t\geq 0)$ and $X''$, we write for conciseness
\[
X' \overset{A}{\oplus} X'' := \big(X'(t) \overset{A}{\oplus} X''(t), \, t\geq 0\big )
\]
For $x=(\pi,\mathbf{v})\in \markpart{\infty}$ and $A\subset \N$, we will say that $A$ is $x$-compatible if it is a union of a family of blocks of $\pi$ -- i.e.\ if $A$ is such that $i\in A,j\notin A \implies i\nsim^{\pi} j$.
These definitions enable us to reformulate the branching property as follows.
\begin{lemma} \label{lem:block_split}
  Let $X$ be an ESSF process, $x=(\pi,\mathbf{v})\in\markpart{\infty}$, and $A\subset \N$ an $x$-compatible set.
  Defining $X':=X^{\sigma_A}$ and $X'':=X^{\sigma_{A^{\mathrm{c}}}}$, then under $\P_x$, $X'$ and $X''$ are two independent copies of the process $X$, respectively started at $x^{\sigma_A}$ and $x^{\sigma_{A^{\mathrm{c}}}}$, and
  \[
  X = X' \overset{A}{\oplus} X''.
  \]
\end{lemma}
\begin{proof}
  This is an immediate consequence of the branching property \textit{(ii)} of Definition \ref{def:ESSF} and of the definition of the $\ssfrag$ operator.
\end{proof}

Let us now tackle the proof of the Markov property for stopping lines \eqref{eqprop:markov_lines}.
We write as usual $X=(\Pi,\mathbf{V})$.
We first assume that there exist $0\leq t_1 < t_2 < \ldots < t_k \leq \infty$ such that for all $i\in\N$, $L_i$ takes values in the finite set $\{t_1, \ldots t_k\}$.
We prove the Markov property for such stopping lines by induction on $k$.
For $k=1$ and $t_1<\infty$, this amounts to the simple Markov property, so \eqref{eqprop:markov_lines} holds by definition.
If $t_1=\infty$, then for all $t\geq 0$, for all $i\geq 1$, $L_i+t\equiv \infty$.
By convention $\mathbf{V}(\infty)$ is the null vector, and by definition $\ssfrag\!\big(X(L),X^{(\cdot)}\big)$ is the process which is a.s.\ constant equal to $X(\infty)=(\Pi(\infty),\mathbf{V}(\infty))$, so \eqref{eqprop:markov_lines} holds again.

Now assume that $k>1$, and that the stopping line Markov property has been proven for all stopping lines taking at most $k-1$ distinct values.
By Remark \ref{rk:stopping_lines}, $L\wedge t_{k-1}$ is a stopping line taking at most $k-1$ distinct values.
Therefore, one can apply the induction hypothesis, which says that conditional on $\mathcal{G}_{L\wedge t_{k-1}}$, the process $X(L\wedge t_{k-1}+\cdot)$ 
has the distribution of a copy of $X$ started from $X(L\wedge t_{k-1})$.
Now we define the random set
\[
A := \{i\in \N, \, L_i = t_{k}\},
\]
which is $\mathcal{G}_{L\wedge t_{k-1}}$-measurable.
Indeed let us show that $\{i\in A\}\in \mathcal{G}_{L\wedge t_{k-1}}$.
Since $\{L_i=t_{k}\}=\{L_i>t_{k-1}\}\in \mathcal{G}_i(t_{k-1})$, one can write the indicator of this event as $\1_{\{L_i>t_{k-1}\}}=F(B_i(s),V_i(s),\,s\in[0,t_{k-1}])$ for a measurable functional $F$ -- recall that $\mathcal{G}_i(t_{k-1})=\sigma(B_i(s),V_i(s),\,s\in[0,t_{k-1}])$, where $B_i(s)$ is the block of $\Pi(s)$ containing $i$.
Now on the event $\{L_i \geq t_{k-1}\}$, we have
\[
F(B_i(s),V_i(s),\,s\in[0,t_{k-1}])=F\big(B_i(s\wedge L_i \wedge t_{k-1}),V_i(s\wedge L_i \wedge t_{k-1}),\,s\in[0,t_{k-1}]\big),
\]
therefore we can write
\[
\{i\in A\} = \{L_i \geq t_{k-1}\}\cap\big\{ F\big(B_i(s\wedge L_i \wedge t_{k-1}),V_i(s\wedge L_i \wedge t_{k-1}),\,s\in[0,t_{k-1}]\big) = 1\big\} \in \mathcal{G}_{L\wedge t_{k-1}},
\]
so finally $A$ is $\mathcal{G}_{L\wedge t_{k-1}}$-measurable.
Now notice that because $L$ is a stopping line, $A$ is compatible with $\Pi(L\wedge t_{k-1})$ in the sense that $A$ is necessarily a union of blocks of $\Pi(L\wedge t_{k-1})$.
Therefore, it is immediate by definition that
\[
X(L\wedge t_{k-1}+\cdot) = X' \overset{A}{\oplus}X'',
\]
with
\begin{align*}
&X' = X(L\wedge t_{k-1}+\cdot)^{\sigma_{A}} \\
\text{ and } &X'' = X(L\wedge t_{k-1}+\cdot)^{\sigma_{A^{\mathrm{c}}}}.
\end{align*}
Now by Lemma \ref{lem:block_split}, 
conditional on $\mathcal{G}_{L\wedge t_{k-1}}$, $X'$ and $X''$ are two independent copies of $X$ started respectively from $X(L\wedge t_{k-1})^{\sigma_{A}}$ and $X(L\wedge t_{k-1})^{\sigma_{A^{\mathrm{c}}}}$.

Also, notice that by definition of the random set $A$, we have the equality
\begin{equation} \label{tmp:eq_stopping_line_markov}
X(L+\cdot) = X'(t_{k}-t_{k-1}+\cdot) \overset{A}{\oplus} X'',
\end{equation}
and for the same reason, the following equality between $\sigma$-algebras holds:
\[
\mathcal{G}_{L} = \mathcal{G}_{L\wedge t_{k-1}} \vee \sigma\big(X'(s),\, s\in [0,t_k-t_{k-1}] \big).
\]
Clearly $X'$ and $X''$ are still independent conditional on $\mathcal{G}_{L}$, and the distribution of $X'(t_{k}-t_{k-1}+\cdot)$ conditional on $\mathcal{G}_{L}$ is by the Markov property at time $t_k-t_{k-1}$ the law of $X$ started from $X'(t_{k}-t_{k-1})=X(L)^{\sigma_{A}}$.
Finally, using again Lemma \ref{lem:block_split}, conditional on $\mathcal{G}_{L}$, the process
\[
X'(t_{k}-t_{k-1}+\cdot) \overset{A}{\oplus} X''
\]
has simply the distribution of a copy of $X$ started at $X(L)$.
So by \eqref{tmp:eq_stopping_line_markov} the Markov property for stopping lines holds for $L$, and so by induction it holds for all stopping lines taking at most a finite number of values.

Now fix a general stopping line $L$, a time $t\geq 0$, and let us assume that our probability space contains an independent sequence $X^{(\cdot)}$ of i.i.d.\ copies of the process started from $(\mathbf{1},1)$.
To conclude, it is enough to prove that conditional on $\mathcal{G}_L$,
\begin{equation} \label{eqproof:one-dim-SLM}
X(L+t) \ed \ssfrag_\alpha\!\big(X(L),X^{(\cdot)}\big)(t),
\end{equation}
because then for any $0 \leq t_1 < t_2 < \ldots t_k$ one can apply successively \eqref{eqproof:one-dim-SLM} to the stopping lines $L+t_i, 1\leq i \leq k$, which implies that \eqref{eqprop:markov_lines} holds for any finite dimensional distributions.
Therefore it remains only to prove
\[
\E\left[F(X(L+t)) Z\right] = \E\left[F\left (\ssfrag_\alpha\!\big(X(L),X^{(\cdot)}\big)(t)\right ) Z\right]
\]
for any fixed continuous bounded map $F:\markpartstar\to\R$ and $Z$ a $\mathcal{G}_{L}$-measurable bounded random variable.
To show this, consider the sequence of stopping lines $(L^{n},n\geq 1)$ defined by
\[
L^{n}_i = 2^{-n}\lceil2^{n}L_i\rceil \1_{L_i \leq n} + \infty \1_{L_i > n}, \quad i\geq 1,
\]
where $\lceil\cdot\rceil$ denotes the usual ceiling function.
This is a classical transformation for stopping times, and it is easily checked that $L^n$ is a stopping line for all $n\geq 1$.
Furthermore, for all $i\geq 1$, $L^n_i$ is nonincreasing and tends to $L_i$ as $n\to\infty$, so since $X$ has right-continuous sample paths, for any $t'\geq 0$, we have $V_i(L^n_i+t') \to V_i(L_i+t')$ and $B_i(L^n_i+t') \to B_i(L_i+t')$ a.s.\ -- recall that $B_i(u)$ denotes the block of $\Pi(u)$ containing $i$.
In other words, for any $t'\geq 0$, the convergence $X(L^n+t')\to X(L+t')$ holds almost surely.
Therefore
\[
\E\left[F(X(L^{n}+t)) Z\right] \underset{n\to\infty}{\longrightarrow} \E\left[F(X(L+t)) Z\right],
\]
Now because $L^{n}$ only takes values in a finite set for all $n$, we can apply \eqref{eqprop:markov_lines}, so
\begin{equation}\label{eqproof:tmp}
\E\left[F(X(L^{n}+t)) Z\right] = \E\left[F\left (\ssfrag_\alpha\!\big(X(L^{n}),X^{(\cdot)}\big)(t)\right ) Z\right].
\end{equation}
This holds because $Z$ is $\mathcal{G}_{L^{n}}$-measurable since $\mathcal{G}_L \subset\mathcal{G}_{L^{n}}$ for all $n\geq 1$.
For the convergence of the right-hand side, recall that $X(L^{n}+t)\to X(L+t)$ and note also that if $V_i(L_i)=0$ for $i\geq 1$, then since $L_i$ is a stopping time, by the strong Markov property $V_i(L_i+t)$ is also zero for all $t\geq 0$, and in particular $V_i(L^{n}_i)=0$ for all $n\geq 1$.
Now by definition an ESSF process is stochastically continuous, so in particular for any $i\geq 1$, on the event $\{V_i(L_i) >0\}$, we have 
\[
X^{(j)}(V_i(L_i)^{\alpha}t) \underset{n\to\infty}{\longrightarrow} X^{(j)}(V_i(L_i)^{\alpha}t) \qquad \forall j\geq 1 \;a.s.
\]
We can now invoke the continuity property of the operator $\ssfrag_\alpha$ pointed out in Remark \ref{rk:fragproc_continuity} and deduce
\[
\ssfrag_\alpha\!\big(X(L^{n}),X^{(\cdot)}\big)(t) \underset{n\to\infty}{\longrightarrow}\ssfrag_\alpha\!\big(X(L),X^{(\cdot)}\big)(t) \qquad a.s.
\]
Taking limits in \eqref{eqproof:tmp} yields the equality needed to end the proof.

\subsection{Proof of Proposition \ref{prop:changing-index}} \label{sec:proof-changing-index}

Let $X=(\Pi,\mathbf{V})$ be an $\alpha$-ESSF process, $t\geq 0$, and recall the definition of $\tau^{\beta}(t)$ as in the proposition, i.e.\
\[
\tau^{\beta}_i(t) \;=\; \left (\int_{0}^{\cdot}V_i(s)^{\beta} \,\d s\right )^{-1}\!\!\!(t) \;=\;
\inf \bigg\{u\geq 0,\,\int_{0}^{u}V_i(s)^{\beta} \,\d s=t\bigg\},
\qquad i\geq 1,
\]
with the convention $\inf \emptyset = \infty$.
For conciseness and because $\beta$ is fixed, let us write simply $\tau$ instead of $\tau^{\beta}$ throughout the proof.
First, let us see that $\tau(t)$ is a stopping line.
Fix $i\in\N$, and note that for $T\geq 0$,
\[
\{\tau_i(t) \leq T\} = \Big\{ \int_{0}^{T}V_{i}(s)^{\beta}\,\d s \geq t \,\Big\} \in \mathcal{G}_i(T),
\]
so $\tau_i(t)$ is a $\mathcal{G}_i$-stopping time.
Furthermore, conditional on $\tau_i(t)$, for any $i,j\in\N$, if $i\sim j$ in $\Pi(\tau_i(t))$, then almost surely for all $0\leq s\leq \tau_i(t)$, we have $i\sim j$ in $\Pi(s)$ so $V_i(s)=V_j(s)$, and necessarily $\tau_j(t) = \tau_i(t)$.

Therefore $\tau(t)$ is indeed a stopping line, so the process $X\circ\tau = (X(\tau(t),t\geq 0)$ is well defined.
We claim that its sample paths are càdlàg in $\markpart{\infty}$.
Indeed, by definition, for each $i\in\N$, $\tau_i$ is a non-decreasing right-continuous map.
Now because $X$ also has càdlàg sample paths by definition, the following holds almost surely: for each $i\in\N$, and $t\in[0,\infty)$,
\[
X(\tau_i(s)) \tol_{s\downarrow t} X(\tau_i(t)) \quad \text{ and }\quad X(\tau_i(s)) \tol_{s\uparrow t} X(\tau_i(t)-) \text{ if } t>0.
\]
Now note that for each stopping line $L$ and integer $n\in\N$, by definition $X(L)\restr{n}$ is a (deterministic) continuous functional of the variables $(X(L_1), \ldots, X(L_n))$.
Applying this to $L=\tau(s)$ and letting $s\to t$, it follows that almost surely
\[
X(\tau(s))\restr{n} \tol_{s\downarrow t} X(\tau(t))\restr{n} \quad \text{ and }\quad X(\tau(s))\restr{n} \text{ converges in }\markpart{n} \text{ when }s\uparrow t, \text{ in the case } t>0.
\]
The integer $n$ being generic, this shows that $X\circ\tau$ is an almost surely càdlàg process.

Let $t\geq 0$ be fixed.
Since $\tau(t)$ is a stopping line, we can apply Proposition \ref{prop:stopline}, and assume that the process $X(\tau(t)+\cdot)$ is given by
\[
\ssfrag_\alpha\!\big(X(\tau(t)), X^{(\cdot)}\big),
\]
where $X^{(\cdot)}$ is an independent sequence of i.i.d.\ copies of the process started from $(\mathbf{1},1)$.
For each $k\in\N$, let $(\tau^{(k)}(s),s\geq 0)$ denote the stopping lines corresponding to $X^{(k)}$, i.e.\ 
\[
\tau_i^{(k)}(s) = \left (\int_{0}^{\cdot}V_i^{(k)}(u)^{\beta} \,\d u\right )^{-1}(s), \qquad i\geq 1, \, s\geq 0.
\]
Our aim is to show that
\begin{equation} \label{eqproof:ssfr_to_mfr}
X(\tau(t+\cdot)) = \ssfrag_{\alpha-\beta}\!\big(X(\tau(t)),X^{(\cdot)}\circ\tau^{(\cdot)}\big).
\end{equation}
Now let us fix $i\in\N$, and work conditional on $\mathcal{G}_{\tau(t)}$.
On the event $\{V_i(\tau_i(t))=0\}$, then by definition of the operators $\ssfrag_\alpha$, the block containing $i$ is constant in time and has mark $0$ in both processes in \eqref{eqproof:ssfr_to_mfr}, so there is equality for index $i$.
Now we condition on $V_i(\tau_i(t))=v$ with $v>0$.
Note that $\{V_i(\tau_i(t))>0\}\subset\{\tau_i(t)<\infty\}$, so in that case we have $\tau_i(t)<\infty$ almost surely, so there is the equality
\[
\int_{0}^{\tau_i(t)}V_i(s)^{\beta} \,\d s = t.
\]
Therefore we can write, for $s\geq 0$,
\begin{align*}
\tau_i(t+s) &= \left (\int_{0}^{\cdot}V_i(u)^{\beta} \,\d u\right )^{-1}(t+s)\\
&= \tau_i(t)+ \left (\int_{0}^{\cdot}V_i(\tau_i(t)+u)^{\beta} \,\d u\right )^{-1}(s).
\end{align*}
Now for all $u\geq 0$, because $X(\tau(t)+\cdot)=\ssfrag_\alpha(X(\tau(t)), X^{(\cdot)})$, we have $V_i(\tau_i(t)+u)=vV^{(k)}_i(v^{\alpha}u)$, for $k$ such that $i$ is in the $k$-th block of $\Pi(\tau(t))$.
This implies
\begin{align*}
\tau_i(t+s)-\tau_i(t) &=\left (\int_{0}^{\cdot}V_i(\tau_i(t)+u)^{\beta} \,\d u\right )^{-1}(s) \\
&= \left (\int_{0}^{\cdot}v^{\beta}V^{(k)}_i(v^{\alpha}u)^{\beta} \,\d u\right )^{-1}(s) \\
&= \left (v^{\beta-\alpha}\!\!\int_{0}^{v^{\alpha}\cdot}V^{(k)}_i(u)^{\beta} \,\d u\right )^{-1}(s) \\
&= v^{-\alpha}\left (\int_{0}^{\cdot}V^{(k)}_i(u)^{\beta} \,\d u\right )^{-1}(v^{\alpha-\beta}s) \\
&= v^{-\alpha}\tau^{(k)}_i(v^{\alpha-\beta}s).
\end{align*}
Note that $L_i(s) := \tau_i(t+s)-\tau_i(t)$ is a stopping time for the natural filtration of the process $V_i(\tau_i(t)+\cdot) = vV_i^{(k)}(v^\alpha \cdot)$, because for all $i\geq 1$ and $u\geq 0$, we have
\[
\{L_i(s) \leq y\} \;=\; \Big\{\int_0^y \big(vV_i^{(k)}(v^\alpha u)\big)^{\beta}\,\d u \geq s \Big\} \;\in\; \sigma\big(vV_i^{(k)}(v^\alpha u), \, u\in [0,y]\big).
\]
This implies that $L(s) = (L_i(s),i\geq 1)$ is a stopping line for the process $X(\tau(t)+\cdot)$, and the equality $X(\tau(t)+L(s))=X(\tau(t+s))$ holds.
We can then write
\[
X(\tau(t+s)) = \ssfrag_\alpha\!\big(X(\tau(t)), X^{(\cdot)}\big)(L(s)),
\]
and compute the mark attached to the block containing $i$ in this marked partition: by definition of the operator $\ssfrag_\alpha$ and using the fact that $L_i(s) = v^{-\alpha}\tau^{(k)}_i(v^{\alpha-\beta}s)$, we have
\begin{align*}
V_i(\tau_i(t+s)) &= v V^{(k)}_i(v^\alpha L_i(s))\\
&=  v V^{(k)}_i(\tau^{(k)}_i(v^{\alpha-\beta}s)).
\end{align*}
We now recognize the mark attached to the block containing $i$ in $\ssfrag_{\alpha-\beta}\!\big(X(\tau(t)), X^{(\cdot)}\circ\tau^{(\cdot)})\big)(s)$.
The same argument allows us to check that $B_i(\tau_i(t+s))$ is also the block of $\ssfrag_{\alpha-\beta}\!\big(X(\tau(t)), X^{(\cdot)}\circ\tau^{(\cdot)})\big)(s)$ containing $i$, so in the end we have exactly
\[
X\circ\tau(t+\cdot) = \ssfrag_{\alpha-\beta}\!\big(X\circ\tau(t), X^{(\cdot)}\circ\tau^{(\cdot)})\big),
\]
which shows that $X\circ \tau$ is an $(\alpha-\beta)$-ESSF.

Finally, note that if $X$ is non-degenerate, then for all $x\in\markpartstar$, $\P_x$-almost surely for any time $t\geq 0$, any block of $X(t)$ is either infinite or has mark $0$.
So $\P_x$-almost surely, for all possible stopping lines $L$, the blocks of $X(L)$ satisfy the same condition, i.e.\ $X(L)\in\markpartstar$.
Therefore, $\P_x$-almost surely $X\circ \tau$ has sample paths with values in $\markpartstar$.

\subsection{Proof of Theorem \ref{thm:main}} \label{sec:proof-main}

Note that the whole process $(X(t),t\geq 0)$ defines a coupling of all $\xi_n$ for $n\geq 1$.
By definition, one has
\begin{equation*}\xi_{n+1}(t) = \xi_n(t), \quad \forall t\leq T_{n+1},
\end{equation*}
and at time $T_{n+1}$, either $\xi_n$ is killed on the event $\{T_n = T_{n+1}\}$, or, conditional on $\{T_n > T_{n+1}\}$, the process $\xi_n$ jumps, independently of the past, according to the probability
\begin{equation*}
\eta_{n+1}(\cdot) := \mathcal{D}_{n+1}\big(\log v_1 \in \cdot \mid \pi\restr{n} = \mathbf{1}_n \text{ and } v_1\neq 0\big),
\end{equation*}
and goes on independently of the past, its remaining part $(\xi_n(T_{n+1}+t)-\xi_n(T_{n+1}),0\leq t\leq T_n - T_{n+1})$ being independent from $\xi_{n+1}$ and equal in distribution to $\xi_n$ by the strong Markov property.
Let us first compute the probability $p_n$ that $T_{n+1} = T_n$, which is by construction
\begin{equation} \label{eqpr:pn-dn}
p_n = \mathcal{D}_{n+1}(\pi\restr{n}\neq \mathbf{1}_n \text{ or }v_1 = 0).
\end{equation}
From the previous description, one can write
\[
T_n = T_{n+1} + ZT'_n,
\]
where $Z = \1_{\{T_{n}\neq T_{n+1}\}}\sim \Be(1-p_n)$ is a Bernoulli random variable with parameter $(1-p_n)$, and $T'_n$ is a random variable equal in distribution to $T_n$, and independent from $T_{n+1}$ and $Z$.
Then, $T_{n+1},Z$ and $T'_n$ are independent because $Z$ is simply a function of the marked partition $D_{n+1}=(\Pi(T_{n+1}),\mathbf{V}(T_{n+1})/V_1(T_{n+1}-))\restr{n+1}$ which is independent from $\xi_{n+1}$, and $T'_n$ is independent of $(\xi_{n+1}, D_{n+1})$ because of the strong Markov property and the fact that $\alpha=0$.
Taking expectations yields
\[
\frac{1}{J_n} = \frac{1}{J_{n+1}} + \frac{1-p_n}{J_n},
\]
which implies that $p_n = J_n / J_{n+1}$.

Now let us rebuild the coupling between $\xi_n$ and $\xi_{n+1}$ to show that their respective Lévy measures $\lambda_n$ and $\lambda_{n+1}$ satisfy
\begin{equation} \label{eqpr:relation-lambda-n}
\lambda_n = \lambda_{n+1} + (J_{n+1}-J_n)\widetilde{\eta}_{n+1},
\end{equation}
where
\begin{equation} \label{eqpr:eta-n-def}
\widetilde{\eta}_{n+1} := \eta_{n+1}\big(\cdot\cap\;\R\!\setminus\!\{0\}\big) = \mathcal{D}_{n+1}\big(\{\log v_1 \in \cdot\}\cap\{v_1\neq 1\} \mid \pi\restr{n} = \mathbf{1}_n \text{ and } v_1\neq 0\big).
\end{equation}
Consider the process $\xi_{n+1}$ a Lévy process with characteristic exponent $\psi_{n+1}$, and let $T_n\sim\Exp(J_n)$ be independent.
Now independently define $T'\sim\Exp(J_{n+1}-J_n)$, and let $T_{n+1} := T_n \wedge T'$.
We see that $T_{n+1}\sim\Exp(J_{n+1})$, that it is independent of $\xi_{n+1}$ and of the event $\{T_n = T_{n+1}\}$, which has probability $J_n/J_{n+1}=p_n$.
Now conditional on $(T_n,T_{n+1})$, let us define
\begin{equation} \label{eq:dnplus1-def}
D_{n+1} = \begin{cases}
D' &\text{if }T_n=T_{n+1}, \quad \text{where } D' \sim \mathcal{D}_{n+1}(\cdot \mid \pi\restr{n}\neq \mathbf{1}_n \text{ or }v_1 = 0)\\
D'' &\text{if }T_n\neq T_{n+1}, \quad \text{where } D'' \sim \mathcal{D}_{n+1}(\cdot \mid \pi\restr{n}= \mathbf{1}_n \text{ and }v_1 \neq 0),
\end{cases}
\end{equation}
where $D'$ and $D''$ are mutually independent and independent of everything else.
Note that $D_{n+1}$ is independent of $T_{n+1}$ and of $\xi_{n+1}$, and because of \eqref{eqpr:pn-dn}, $D_{n+1}$ has indeed distribution $\mathcal{D}_{n+1}$.
Let us define $J=\log v_1$, where $v_1$ is the mark associated with the integer $1$ in the marked partition $D_{n+1}$, and $\xi_n$ a Lévy process with characteristic exponent $\psi_n$.
Now putting everything together, define $\tilde{\xi}_{n+1}$ as the killed Lévy process $(\xi_{n+1}(t), 0\leq t < T_{n+1})$, and define $\tilde{\xi}_n$ as
\[
\tilde{\xi}_n(t) = \begin{cases}
\xi_{n+1}(t) &\text{ if }t<T_{n+1}\\
\xi_{n+1}(T_{n+1})+J + \xi_n(t-T_{n+1}) & \text{ if }T_{n+1}\leq t < T_n.
\end{cases}
\]
By construction, the joint distribution of $(\tilde{\xi}_{n}, \tilde{\xi}_{n+1})$ is equal to the one we get from the original process $X$, and it should now be clear that the point process of jumps of $\xi_n$ is equal in distribution to the point process of jumps of $\xi_{n+1}$ with additional jumps distributed as $J=\log v_1$, arising at rate $(J_{n+1}-J_n)$.
Note that by construction, $J$ has distribution $\eta_{n+1}$, so finally we have proven \eqref{eqpr:relation-lambda-n}.
The fact that $(\lambda_n,n\geq 1)$ is a nonincreasing sequence of $\sigma$-finite measures ensures the existence of a limiting measure $\lambda_\infty$ on $\R\setminus\{0\}$ such that for all $n\in\N$,
\begin{equation} \label{eq:lambda-infty-lambda-n}
\lambda_n = \lambda_\infty + \sum_{k>n}(J_{k}-J_{k-1})\widetilde{\eta}_k.
\end{equation}
Recall that we wrote the characteristic exponent of $\xi_n$ in the following way:
\[
\psi_n(\theta) = id_n\theta -\frac{\beta_n}{2}\theta^2 + \int_{\R}\Big(\e^{i\theta y} -1-i\theta y\1_{\abs{y}\leq 1} \Big) \,\lambda_n(\d y).
\]
From the previous discussion, one can construct a coupling between the two Lévy processes such that $(\xi_n(t)-\xi_{n+1}(t),t\geq 0)$ is simply a compound Poisson process with jump measure $(J_{n+1}-J_n)\widetilde{\eta}_{n+1}$ which is independent from $\xi_{n+1}$.
It is then clear that necessarily $\beta_{n}=\beta_{n+1}$, and
\[
d_n = d_{n+1} + (J_{n+1}-J_n)\int_{\abs{y}\leq 1}y\,\widetilde{\eta}_{n+1}(\d y) = d_{n+1} + \int_{\abs{y}\leq 1}y\,(\lambda_n-\lambda_{n+1})(\d y).
\]
To summarize, letting $\beta := \beta_1$, the following holds for all $n\in\N$
\begin{equation} \label{eq:betan_dn}
\beta_n=\beta \quad \text{ and } \quad d_n = d_1 - \int_{\abs{y}\leq 1}y\,(\lambda_1-\lambda_n)(\d y),
\end{equation}
where $(\lambda_1 - \lambda_n)$ denotes the positive measure given by
\[
(\lambda_1 - \lambda_n) = \sum_{k=2}^{n}(J_{k}-J_{k-1})\widetilde{\eta}_k.
\]
Let us now examine the consistency properties of the measures $\mathcal{D}_n$.
From this point on, for the sake of clarity, we decompose the proof in a series of steps.

\paragraph{Step 1.} We prove the existence and uniqueness of a measure $\mathcal{D}$ on $\markpart{\infty}$ satisfying
\begin{equation} \label{eq:d-support}
\mathcal{D}(\pi=\mathbf{1}\text{ and }v_1\neq 0) = 0
\end{equation}
and such that for all $n\in\N$,
\begin{equation} \label{eq:consistency-D-Dn}
\mathcal{D}\big(\{x\restr{n}\in\cdot\}\cap\{\pi\restr{n}\neq \mathbf{1}_n \text{ or }v_1 = 0\}\big) = J_n \mathcal{D}_n,
\end{equation}
then we show that this measure is exchangeable.

First, note that for the construction above to be consistent
-- with $D_{n+1}$ defined in \eqref{eq:dnplus1-def} as $D_{n+1} = D'$ on the event $\{T_n=T_{n+1}\}$, where $D'\sim \mathcal{D}_{n+1}(\cdot \mid \pi\restr{n}\neq \mathbf{1}_n \text{ or }v_1 = 0)$ --
the random variable $D'\restr{n}$ must have distribution $\mathcal{D}_n$.
Indeed, on the event $\{T_{n+1}<T_n\}$, the strong Markov property at time $T_{n+1}$ implies that the process $X\restr{n}$ jumps according to $\mathcal{D}_n$, independently of the past, so on the complement this must hold as well, so
\[
\mathcal{D}_{n+1}\big(x\restr{n}\in \cdot \mid \pi\restr{n}\neq \mathbf{1}_n \text{ or }v_1 = 0\big) = \mathcal{D}_n,
\]
which can be rewritten
\begin{equation} \label{eq:compat-dn-n+1}
J_{n+1}\mathcal{D}_{n+1}\big(\{x\restr{n}\in\cdot\}\cap\{\pi\restr{n}\neq \mathbf{1}_n \text{ or }v_1 = 0\}\big) = J_n \mathcal{D}_n.
\end{equation}
Now for all integers $n\leq m$, let us define a measure on $\markpart{m}$ by
\[
\mu_{n}^{m} := J_m \mathcal{D}_m\big( \cdot\cap \{\pi\restr{n}\neq \mathbf{1}_n \text{ or }v_1 = 0\}\big), \quad m\geq n.
\]
Let us prove that for all integers $n\leq k \leq m$, we have
\begin{equation} \label{eq:compat-mu-n-m-k}
\mu_{n}^{m}(x\restr{k} \in \cdot) = \mu_{n}^{k}.
\end{equation}
Note that there is nothing to prove in the case $k=m$.
Now suppose this is proven for fixed $n\leq k\leq m$.
Then,
\begin{align*}
\mu_{n}^{m+1}(x\restr{k}\in\cdot) &= J_{m+1}\mathcal{D}_{m+1}\big(\{x\restr{k}\in\cdot\}\cap\{\pi\restr{n}\neq \mathbf{1}_n \text{ or }v_1 = 0\}\big)\\
&= J_{m+1}\mathcal{D}_{m+1}\big(\{(x\restr{m})\restr{k}\in\cdot\} \cap\{(\pi\restr{m})\restr{n}\neq \mathbf{1}_n \text{ or }v_1 = 0\} \\
& \phantom{=J_{m+1}\mathcal{D}_{m+1}\big(\{(x\restr{m})\restr{k}\in\cdot\}} \;\cap\{\pi\restr{m}\neq \mathbf{1}_m \text{ or }v_1 = 0\}\big)\\
&= J_m \mathcal{D}_m\big( \{x\restr{k}\in\cdot\} \cap\{\pi\restr{n}\neq \mathbf{1}_n \text{ or }v_1 = 0\}\big)\\
&= \mu_n^{m}(x\restr{k}\in\cdot) = \mu_n^k,
\end{align*}
where we have used \eqref{eq:compat-dn-n+1} and the fact that $\{\pi\restr{n}\neq \mathbf{1}_n \text{ or }v_1 = 0\}\subset \{\pi\restr{m}\neq \mathbf{1}_m \text{ or }v_1 = 0\}$.
By induction on $m$ this proves \eqref{eq:compat-mu-n-m-k} for any integers $n\leq k\leq m$.
Note that in particular, taking $k=n$, we see that the total mass of $\mu_n^{m}$ is equal to that of $\mu_n^{n}$, which is $J_n$.
In summary, for any $n\in\N$, the sequence $(\markpart{m},\mu_n^{m}/J_n,m\geq n)$ defines a inverse system of compact probability spaces, and by the Kolmogorov extension theorem, there exists a unique measure $\mu_n$ (with total mass $J_n$) on the inverse limit $\varprojlim_{m}\markpart{m} = \markpart{\infty}$ such that for each $m\geq n$, $\mu_n(x\restr{m}\in\cdot)=\mu_{n}^{m}$.
Now notice that by definition, for any integers $n_1 \leq n_2 \leq m$, we have
\[
\mu_{n_2}^{m}(\cdot \cap \{\pi\restr{n_1}\neq \mathbf{1}_{n_1} \text{ or }v_1 = 0\}) = \mu_{n_1}^{m},
\]
which implies by construction
\[
\mu_{n_2}(\cdot \cap \{\pi\restr{n_1}\neq \mathbf{1}_{n_1} \text{ or }v_1 = 0\}) = \mu_{n_1}.
\]
This means that the sequence of measures $(\mu_n, n\geq 1)$ on $\markpart{\infty}$ is increasing, and one can define the limit as $\mathcal{D}$.
This measure then satisfies by construction
\begin{align*}
\mathcal{D}\big(\{x\restr{n}\in\cdot\}\cap\{\pi\restr{n}\neq \mathbf{1}_n \text{ or }v_1 = 0\}\big) &= \mu_{n}(x\restr{n}\in\cdot) \\
&= \mu_n^{n} \\
&= J_n \mathcal{D}_n,
\end{align*}
which is indeed \eqref{eq:consistency-D-Dn}.

Secondly, note that since for any $n\in\N$, clearly $\mu_n(\pi=\mathbf{1}\text{ and }v_1\neq 0) = 0$, where $\mu_n$ are the measures defined above, so in the limit \eqref{eq:d-support} holds.
Let us now show uniqueness.
If a measure $\mathcal{D}'$ on $\markpart{\infty}$ satisfies \eqref{eq:consistency-D-Dn} and \eqref{eq:d-support} then
\[
\mathcal{D}'(\cdot\cap\{\pi\restr{n}\neq \mathbf{1}_{n} \text{ or }v_1 = 0\}) = \mu_{n},
\]
and letting $n\to\infty$, $\mathcal{D}'=\mathcal{D}$, which proves uniqueness.

Finally, $\mathcal{D}$ is exchangeable.
Indeed if $\sigma:\N\to\N$ is a permutation, let $m\in\N$ such that $\sigma(k)=k$ for all $k\geq m$.
Now for all $n\geq m$, using the exchangeability of the probability measures $(\mathcal{D}_k,k\geq 1)$, we get
\begin{align*}
\mathcal{D}((x^{\sigma})\restr{n}\in\cdot) &= \lim_{k\to\infty}\mathcal{D}\big(\{(x^{\sigma})\restr{n}\in\cdot\}\cap\{\pi\restr{k}\neq \mathbf{1}_{k} \text{ or }v_1 = 0\}\big)\\
&= \lim_{k\to\infty}J_k\mathcal{D}_k((x^{\sigma})\restr{n}\in\cdot)\\
&= \lim_{k\to\infty}J_k\mathcal{D}_k(x\restr{n}\in\cdot)\\
&= \lim_{k\to\infty}\mathcal{D}(\{x\restr{n}\in\cdot\}\cap\{\pi\restr{k}\neq \mathbf{1}_{k} \text{ or }v_1 = 0\})\\
&= \mathcal{D}(x\restr{n}\in\cdot).
\end{align*}
As this is true for all $n\geq m$, necessarily $\mathcal{D}(x^{\sigma}\in\cdot)=\mathcal{D}$, i.e.\ $\mathcal{D}$ is exchangeable.

\paragraph{Step 2.} We prove that $\mathcal{D}(\markpart{\infty}\setminus\markpartstar)=0$ by using that $X$ is non-degenerate.
For this, we need to show first that the process $(B_1(t),t\geq 0)$ of the block of $X$ containing $1$ is equal in distribution to a process $(B(t),t\geq 0)$ constructed from a Poisson point process of intensity $\d t \otimes \mathcal{D}$.

More precisely, define $\mathcal{N}$ a Poisson point process on $[0,\infty)\times\markpart{\infty}$ with intensity $\d t\otimes (\mathcal{D}+\widetilde{\lambda}_\infty)$, where $\widetilde{\lambda}_\infty$ is the push-forward of $\lambda_\infty$ by the map $y \in\R \mapsto (\mathbf{1},\e^{y}) \in\markpartstar$.
Let us define
\[
\mathcal{N}' := \big\{(t,\log v_1), \; (t,x)\in\mathcal{N} \text{ with }x =(\pi,\mathbf{v}) \text{ and } v_1\notin \{0,1\}\big\},
\]
which is then a Poisson point process on $[0,\infty)\times \R$ with intensity
\[
\d t \otimes \Big(\mathcal{D} \big(\{\log v_1\in\cdot\}\cap\{v_1\notin \{0,1\}\}\big) + \lambda_\infty \Big).
\]
However, note that using \eqref{eq:d-support}, \eqref{eq:consistency-D-Dn} and finally \eqref{eqpr:eta-n-def}, we get
\begin{align*}
\mathcal{D}&\big(\{\log v_1\in\cdot\}\cap\{v_1\notin \{0,1\}\}\big)\\
&= \sum_{n\in\N}\mathcal{D}\big(\{\log v_1\in\cdot\}\cap\{\pi\restr{n} = \mathbf{1}_n\text{ and }v_1\notin \{0,1\}\}\cap\{\pi\restr{n+1} \neq \mathbf{1}_{n+1}\}\big)\\
&= \sum_{n\in\N}J_{n+1}\mathcal{D}_{n+1}\big(\{\log v_1\in\cdot\}\cap\{\pi\restr{n} = \mathbf{1}_n\text{ and }v_1\notin \{0,1\}\}\big)\\
&= \sum_{n\in\N}(J_{n+1}-J_n) \widetilde{\eta}_{n+1},
\end{align*}
and by \eqref{eq:lambda-infty-lambda-n}, we find $\mathcal{D} \big(\{\log v_1\in\cdot\}\cap\{v_1\notin \{0,1\}\}\big) + \lambda_\infty = \lambda_1$.
Therefore, it is possible to define a Lévy process $\xi$ with characteristic exponent $\psi_1$, such that the point process of its jumps is precisely $\mathcal{N}'$.
Let us also define a process $B=(B(t),t\geq 0)$ with càdlàg sample paths with values in $2^{\N}$ the subsets of $\N$, such that $B$ has the distribution of $(B_1(t),t\geq 0)$ the block containing $1$ in $X$.
First define $T$ as the first time $t\in[0,\infty)$ such that there is an atom $(t,(\pi,\mathbf{v}))\in \mathcal{N}$ with $v_1=0$.
If there is none, then let $T=\infty$.
Then, for each $n\in\N$, let $(t_1,x_1),(t_2,x_2),\ldots$ be the whole sequence (finite or infinite) of atoms of $\mathcal{N}$ with $t_1 < t_2 < \ldots \leq T$ such that for each $i$, $x_i=(\pi_i,\mathbf{v}_i)$ with $(\pi_i)\restr{n}\neq \mathbf{1}_n$ and $(v_i)_1>0$ (or such that $(v_i)_1=0$ for the possible last atom, at time $T$).
We can define $\widetilde{B}^{n}_0 = [n]$, and inductively for each $i\geq 1$
\[
\widetilde{B}^{n}_{i} := \widetilde{B}^{n}_{i-1} \cap (A_i \cap [n]),
\]
where $A_i$ is the block of $\pi_i$ containing $1$.
Now define, for $t\in[0,\infty)$,
\[
B^{n}(t) = \widetilde{B}^{n}_i \text{ if }t\in[t_i,t_{i+1}),
\]
where we let $t_0:=0$, and in the case $T < \infty$, i.e.\ if the sequence of atoms is finite, say with length $k\in\N$, we let $t_{k+1} := \infty$.
It is readily checked that this construction is consistent in the sense that for each $t\geq 0$ there is a single $B(t)\in 2^{\N}$ such that $B^{n}(t)=B(t)\cap[n]$.
Let us show that this process $(B(t),\xi(t),t\geq 0)$ has the same distribution as the marked block containing $1$ in $X$, i.e.\ $(B_1(t), \log V_1(t),t\geq 0)$.
For fixed $n\in\N$ and $x\in\markpart{n}$, recall that
\[
X\restr{n} \text{ under } \P_{x} \ed \ssfrag_0(x,X^{(\cdot)}),
\]
where $X^{(\cdot)}$ is an independent i.i.d.\ sequence of copies of $X$ started from $(\mathbf{1},1)$.
Using the same notation, for any $A\subset [n]$ with $1\in A$ and $v>0$, the law of the process $(B_1(t)\cap[n],\log V_1(t),t\geq 0)$ started from $(A,\log v)$ can be deduced from that of $X^{(1)}=(\Pi^{(1)},\mathbf{V}^{(1)})$.
More precisely, $\log V_1(t)$ behaves as a Lévy process with characteristic exponent $\psi_n$ started from $\log v$, until an independent time $T_n\sim\Exp(J_n)$ when $\Pi^{(1)}\restr{n}$ first jumps.
At that time, $D_n^{(1)}=(\Pi^{(1)}(T_n),\mathbf{V}^{(1)}(T_n)/V_1(T_n-))\restr{n}$ is independently drawn according to $\mathcal{D}_n$ and then writing $D_n^{(1)}=(\pi,\mathbf{v})$,
\[
\log V_1(T_n) = \log V_1(T_n-)+\log v_1 \quad\text{ and }\quad B_1(T_n) = B_1(T_n-)\cap A_1=A\cap A_1,
\]
where $A_1$ is the block of $\pi$ containing $1$.
Note that there is a non-zero probability that $B_1(T_n) = B_1(T_n-)$ (even with $v_1>0$) when $A\neq [n]$.
Now in our construction, if $(t_1,x)$ is the first atom of $\mathcal{N}$ such that $\pi\neq \mathbf{1}_n$ or $v_1 = 0$, where $x=(\pi,\mathbf{v})$, then $t_1$ is exponentially distributed with parameter $\mathcal{D}(\pi\restr{n}\neq \mathbf{1}_n \text{ or }v_1 = 0) = J_n$, and $x$ is independent of $t_1$, distributed as
\[
\frac{1}{J_n}\mathcal{D}(\cdot \cap \{\pi\neq \mathbf{1}_n \text{ or }v_1 = 0\} ),
\]
so that $x\restr{n}$ has distribution $\mathcal{D}_n$.
It remains only to show that $(\xi(s),0\leq s< t_1)$ is distributed as a Lévy process with characteristic exponent $\psi_n$ (killed at $t_1$).
The point process of its jumps is
\[
\mathcal{N}'\cap [0,t_1)\times \R,
\]
which conditional on $t_1$ has intensity 
\[
\d t\otimes\Big(\mathcal{D}\big(\{\log v_1\in\cdot\}\cap\{\pi\restr{n} = \mathbf{1}_n\text{ and }v_1\notin \{0,1\}\}\big)+\lambda_\infty\Big).
\]
Let us show that this intensity is equal to $\d t\otimes\lambda_n$.
Note that
\begin{align*}
&\mathcal{D}\big(\{\log v_1\in\cdot\}\cap\{\pi\restr{n} = \mathbf{1}_n\text{ and }v_1\notin \{0,1\}\}\big)\\
&= \sum_{m\geq n}\mathcal{D}\big(\{\log v_1\in\cdot\}\cap\{\pi\restr{m} = \mathbf{1}_m\text{ and }v_1\notin \{0,1\}\}\cap\{\pi\restr{m+1} \neq \mathbf{1}_{m+1}\}\big)\\
&= \sum_{m\geq n}J_{m+1}\mathcal{D}_{m+1}\big(\{\log v_1\in\cdot\}\cap\{\pi\restr{m} = \mathbf{1}_m\text{ and }v_1\notin \{0,1\}\}\big)\\
&= \sum_{m\geq n}(J_{m+1}-J_m) \widetilde{\eta}_{m+1},
\end{align*}
so \eqref{eq:lambda-infty-lambda-n} shows that
\begin{equation} \label{eq:lambda-n-infty-with-D}
\lambda_n=\lambda_\infty+\mathcal{D}\big(\{\log v_1\in\cdot\}\cap\{\pi\restr{n} = \mathbf{1}_n\text{ and }v_1\notin \{0,1\}\}\big)
\end{equation}
therefore the Lévy measure of $(\xi(s),0\leq s< t_1)$ is indeed $\lambda_n$.
In the end, we have shown that
\[
(B(t),\xi(t),t\geq 0)\ed (B_1(t), \log V_1(t),t\geq 0).
\]
From this construction, we see that for each atom $(t,x)\in \mathcal{N}$, the process $B$ jumps, with
\[
B(t) = B(t-)\cap A,
\]
where $A$ is the block of $x$ containing $1$.
Let us show that this implies $\mathcal{D}(\markpart{\infty}\setminus\markpartstar)=0$.
Assuming the opposite, there is a non-zero probability that there is an atom $(t,x)\in \mathcal{N}$ with $t< T$ such that $x$ contains a finite block with mark not equal to zero.
By exchangeability of $\mathcal{D}$, and from the description of the jumps of $B$, there is a non-zero probability that there is a jump $B(t) = B(t-) \cap A$ where $A$ is finite and $t < T$.
This contradicts the assumption of non-degeneracy of $X$, as then we would have $X(t)\in \markpart{\infty}\setminus\markpartstar$.

From now on, we view $\mathcal{D}$ as an exchangeable measure on $\markpartstar$, satisfying \eqref{eq:d-support} and the $\sigma$-finiteness assumption
\[
\forall n\in\N, \quad \mathcal{D}\big(\pi\restr{n}\neq \mathbf{1}_n\text{ or } v_1=0\big) = J_n < \infty.
\]
It remains essentially to study $\mathcal{D}$ in order to express it as a mixture of paintbox processes.

\paragraph{Step 3.} Let us decompose
\[
\mathcal{D} = \mathcal{D}\big(\cdot\cap\{\pi=\mathbf{1}\}\big)+\mathcal{D}\big(\cdot\cap\{\pi\neq \mathbf{1}\text{ and }\abs{\pi}^{\downarrow}=\mathbf{1}\}\big)+\mathcal{D}\big(\cdot\cap\{\pi\neq \mathbf{1}\text{ and }\abs{\pi}^{\downarrow}\neq \mathbf{1}\}\big)\\
\]
and show that there exist $c\geq 0$ a constant and $\Lambda'$ a $\sigma$-finite measure on $\markmasspart$ such that
\begin{enumerate}[(a)]
  \item \label{eq:d-decomp1}
  $\mathcal{D}\big(\cdot\cap\{\pi=\mathbf{1}\}\big)=\mathcal{D}(\pi=\mathbf{1})\delta_{(\mathbf{1},0)},$
  \item \label{eq:d-decomp2}
  $\displaystyle\mathcal{D}\big(\cdot\cap\{\pi\neq \mathbf{1}\text{ and }\abs{\pi}^{\downarrow}=\mathbf{1}\}\big)=c \sum_{n\in\N}\delta_{\mathfrak{e}_n},$
  where
  \[
  \mathfrak{e}_n := \Big(\big\{\{n\},\N\setminus\{n\}\big\},\, (1,\ldots,1,\underbrace{0}_{n-\text{th index}},1,\ldots) \Big) \in \markpartstar,
  \]
  \item \label{eq:d-decomp3}
  $\displaystyle\mathcal{D}\big(\cdot\cap\{\pi\neq \mathbf{1}\text{ and }\abs{\pi}^{\downarrow}\neq \mathbf{1}\}\big)= \int_{\markmasspart}\rho_{\mathbf{z}}(\cdot)\,\Lambda'(\d \mathbf{z}),$
\end{enumerate}

We use similar arguments as in \cite[Theorem 3.1]{Ber06}, as we have already done in the context of nested fragmentations \cite[Proposition 19]{Duc18}.
First note that by \eqref{eq:d-support}, $\mathcal{D}$-a.e.\ on the event $\{\pi=\mathbf{1}\}$ we have $x=(\mathbf{1},0)$, in other words \ref{eq:d-decomp1} holds.
Note also that $\mathcal{D}(\pi=\mathbf{1})\leq \mathcal{D}(v_1=0)=J_1<\infty$.

Let us now study the measure $\mathcal{D}(\cdot\cap\{\pi\neq \mathbf{1}\})$.
Note that $\mathcal{D}(\{\pi\in\cdot\}\cap\{\pi\neq \mathbf{1}\})$ is an exchangeable measure on $\partit{\infty}$ satisfying for all $n\geq 1$,
\[
\mathcal{D}(\pi\restr{n}\neq \mathbf{1}_n) < \infty.
\]
A consequence of \cite[Theorem 3.1]{Ber06} is that $\pi$ has asymptotic frequencies $\mathcal{D}$-a.e.\ -- recall that $\abs{\pi}^{\downarrow}\in\masspart\subset[0,1]^{\N}$ denotes the nonincreasing reordering of asymptotic frequencies of blocks of $\pi$ -- and one can write
\[
\mathcal{D}(\{\pi\in\cdot\}\cap\{\pi\neq \mathbf{1}\text{ and }\abs{\pi}^{\downarrow} = (1,0,0,\ldots)\}) = c \sum_{n\in\N}\delta_{\{\{n\},\N\setminus\{n\}\}},
\]
where $c\geq 0$.
For conciseness -- and again with some abuse of notation -- we will from now on let $\mathbf{1}:=(1,0,0\ldots)\in\masspart$.
Now let us examine the distribution of $x=(\pi,\mathbf{v})$ on the event $\big\{\pi=\{\{n\},\N\setminus\{n\}\}\big\}$.
Since $x\in\markpartstar$ $\mathcal{D}$-a.e., the singleton block must have mark $0$, while the other block may have a positive mark.
Let $\eta$ be the distribution of this mark on $S^1$, that is
\[
\eta := \mathcal{D}\big(\{v_1\in\cdot\}\cap\big\{\pi=\{\{n\},\N\setminus\{n\}\}\big\}\big),
\]
which is a measure of total mass $c$, for any fixed $n>1$ (by exchangeability, $\eta$ does not depend on the value of $n$).
First, note that $\eta(\{0\}) = 0$.
Indeed, since the events
\[
\big\{\pi=\{\{n\},\N\setminus\{n\}\}\big\}, \qquad n > 1
\]
are disjoint, the following holds.
\begin{align*}
\sum_{n>1} \eta(\{0\}) &= \sum_{n>1}\mathcal{D}\big(\big\{v_1=0\text{ and }\pi=\{\{n\},\N\setminus\{n\}\}\big\}\big) \\
&= \mathcal{D}\Big(\bigcup_{n>1}\big\{v_1=0\text{ and }\pi=\{\{n\},\N\setminus\{n\}\}\big\}\Big)\\
&\leq \mathcal{D}(v_1=0)\\
&=J_1 < \infty,
\end{align*}
which implies necessarily $\eta(\{0\}) = 0$.
Now recall that $\mathcal{D}\big(\{\log v_1\in\cdot\}\cap\{v_1 \notin \{0,1\}\}\big)$ is a Lévy measure, so for all $\epsilon > 0$,
\begin{align*}
\sum_{n>1} \eta(\abs{\log v_1} > \epsilon) &= \sum_{n>1}\mathcal{D}\big(\big\{\abs{\log v_1} > \epsilon\text{ and }\pi=\{\{n\},\N\setminus\{n\}\}\big\}\big) \\
&= \mathcal{D}\Big(\bigcup_{n>1}\big\{\abs{\log v_1} > \epsilon\text{ and }\pi=\{\{n\},\N\setminus\{n\}\}\big\}\Big)\\
&\leq \mathcal{D}\big(\abs{\log v_1}>\epsilon\text{ and }v_1 \neq 0\big) < \infty,
\end{align*}
which implies necessarily $\eta(\abs{\log v_1} > \epsilon) = 0$.
Letting $\epsilon\to 0$, we have $\eta(\abs{\log v_1} > 0) = 0$, so in the end $\eta = c\delta_1$, and \ref{eq:d-decomp2} follows.

Let us now decompose the measure $\mathcal{D}(\cdot\cap\{\abs{\pi}^{\downarrow} \neq \mathbf{1}\})$.
Recall that by construction,
\[
\mathcal{D}(\cdot\cap \{\pi\restr{n}\neq \mathbf{1}\text{ and }\abs{\pi}^{\downarrow}\neq \mathbf{1}\})\leq \mathcal{D}(\cdot\cap \{\pi\restr{n}\neq \mathbf{1}\text{ or }v_1=0\})=\mu_n,
\]
which is a finite measure with total mass $J_n$.
Now let us introduce the injection $\theta_n:\N\to\N,k\mapsto n+k$, and consider
\[
\overset{\leftarrow}{\mu}_n := \mathcal{D}\big(\{x^{\theta_n}\in\cdot\}\cap\{\pi\restr{n}\neq \mathbf{1}\text{ and }\abs{\pi}^{\downarrow}\neq \mathbf{1}\}\big),
\]
which can be seen as the distribution of the marked partition restricted to $\{n+1,n+2,\ldots\}$, on the event $\{\pi\restr{n}\neq \mathbf{1}\text{ and }\abs{\pi}^{\downarrow}\neq \mathbf{1}\}$.
It is readily checked that this measure is exchangeable on $\markpartstar$ and since it is finite, by Proposition \ref{prop:paintbox},
\[
\overset{\leftarrow}{\mu}_n = \int_{\markmasspart}\rho_{\mathbf{z}}(\cdot)\,\Lambda_n(\d \mathbf{z}),
\]
with $\Lambda_n = \overset{\leftarrow}{\mu}_n(\abs{x}^{\downarrow}\in\cdot)$ a finite measure on $\markmasspart$.
Asymptotic frequencies are such that $\abs{x}^{\downarrow}=\abs{x^{\theta_n}}^{\downarrow}$ for all $x\in\markpartstar$, therefore $\Lambda_n$ can also be written
\[
\Lambda_n = \mathcal{D}\big(\{\abs{x}^{\downarrow}\in\cdot\}\cap\{\pi\restr{n}\neq \mathbf{1}\text{ and }\abs{\pi}^{\downarrow}\neq \mathbf{1}\}\big),
\]
and taking nondecreasing limits one can define\[
\Lambda' := \mathcal{D}\big(\{\abs{x}^{\downarrow}\in\cdot\}\cap\{\pi\neq \mathbf{1}\text{ and }\abs{\pi}^{\downarrow}\neq \mathbf{1}\}\big).
\]
To show \ref{eq:d-decomp3}, fix $k,n\in\N$ and consider the permutation $\tau:\N\to\N$ given by
\[
\tau(i) = \begin{cases}
i+k&\text{ if }i\leq n\\
i-n&\text{ if }n<i\leq n+k\\
i&\text{ otherwise.}
\end{cases}
\]
Now notice that
\[
\mathcal{D}\big(\{x\restr{k}\in\cdot\}\cap\{\pi\neq \mathbf{1}\text{ and }\abs{\pi}^{\downarrow}\neq \mathbf{1}\}\big) = \lim_{n\to\infty}\mathcal{D}\big(\{x\restr{k}\in\cdot\}\cap\{(\pi^{\theta_k})\restr{n}\neq \mathbf{1}_n\text{ and }\abs{\pi}^{\downarrow}\neq \mathbf{1}\}\big),
\]
which can be written, using the exchangeability of $\mathcal{D}$,
\begin{align*}
&\mathcal{D}\big(\{x\restr{k}\in\cdot\}\cap\{(\pi^{\theta_k})\restr{n}\neq \mathbf{1}_n\text{ and }\abs{\pi}^{\downarrow}\neq \mathbf{1}\}\big)\\
&=\mathcal{D}\big(\{(x^{\tau\circ\theta_n})\restr{k}\in\cdot\}\cap\{(\pi^{\tau})\restr{n}\neq \mathbf{1}_n\text{ and }\abs{\pi^{\tau}}^{\downarrow}\neq \mathbf{1}\}\big)\\
&=\mathcal{D}\big(\{(x^{\theta_n})\restr{k}\in\cdot\}\cap\{\pi\restr{n}\neq \mathbf{1}_n\text{ and }\abs{\pi}^{\downarrow}\neq \mathbf{1}\}\big)\\
&=\int_{\markmasspart}\rho^k_{\mathbf{z}}(\cdot)\,\Lambda_n(\d \mathbf{z}),
\end{align*}
where $\rho_{\mathbf{z}}^{k}$ is the paintbox process restricted to $k$ elements defined in Section \ref{sec:paintbox}.
Taking limits and because $k$ is generic, we have indeed \ref{eq:d-decomp3}.

\paragraph{Step 4.} It remains to define the measure $\Lambda$ correctly and we will be able to complete the proof of Theorem~\ref{thm:main}.
Recall the definition of $\widetilde{\lambda}_\infty$ as the push-forward of $\lambda_\infty$ by the map $y \in\R \mapsto (\mathbf{1},\e^{y}) \in\markpartstar$, and note that
\[
\widetilde{\lambda}_\infty = \int_{\markmasspart}\rho_{\mathbf{z}}(\cdot)\,\widehat{\lambda}_\infty(\d \mathbf{z}),
\]
where $\widehat{\lambda}_\infty$ is the push-forward of $\lambda_\infty$ by the map $y \in\R \mapsto (\mathbf{1},\e^{y}) \in\markmasspart$.
In the end, let us define
\begin{equation*}
\Lambda = \Lambda'+\mathcal{D}(\pi=\mathbf{1})\delta_{(\mathbf{1},0)}+\widehat{\lambda}_\infty.
\end{equation*}
Putting everything together, we have
\[
\mathcal{D}+\widetilde{\lambda}_\infty =  c \sum_{n\in\N}\delta_{\mathfrak{e}_n}+\int_{\markmasspart}\rho_{\mathbf{z}}(\cdot)\,\Lambda(\d \mathbf{z}).
\]
We can almost complete the proof.
First, fix $n\in\N$ and recall that by definition $J_n = \mathcal{D}(\pi\restr{n}\neq\mathbf{1}_n\text{ or }v_1=0)$.
Since by definition $\widetilde{\lambda}_\infty(\pi\restr{n}\neq\mathbf{1}_n\text{ or }v_1=0)=0$ for all $n$, point (ii) of Theorem \ref{thm:main} is proven:
\[
J_n = nc + \int_{\markmasspart}\Big (1-\sum_{\substack{i\geq 1\\ v_i>0}}s_i^n\Big )\,\Lambda(\d \mathbf{z}).
\]
This implies that $\big(1-\sum_{\substack{i\geq 1\\ v_i>0}}s_i^n\big)$ is $\Lambda$-integrable for any $n\geq 1$.
Furthermore, note that for any $\mathbf{z}=(\mathbf{s},\mathbf{v})\in\markmasspart$,
\[
1-s_1\leq 1-s_1\Big(\sum_{\substack{i\geq 1\\v_i>0}}s_i\Big)\leq 1-\sum_{\substack{i\geq 1\\v_i>0}}s_i^2, \quad \text{and} \quad s_1\1_{v_1=0} \leq 1-\sum_{\substack{i\geq 1\\v_i>0}}s_i,
\]
therefore summing the two expressions yields
\begin{equation}\label{eqpr:integ-1-s1}
\int_{\markmasspart}(1-s_1\1_{v_1>0})\,\Lambda(\d \mathbf{z}) < \infty.
\end{equation}
We keep this in mind for later use and go back to the expression of the measures $\mathcal{D}_n$.
If $J_n>0$, then $\mathcal{D}_n=\mathcal{D}\big(\{x\restr{n}\in\cdot\}\cap\{\pi\restr{n}\neq\mathbf{1}_n\text{ or }v_1=0\}\big)/J_n$, so point (iii) of Theorem \ref{thm:main} is proven:
\[
\mathcal{D}_n=\frac{1}{J_n}\Big(\sum_{i=1}^{n}c\delta_{\mathfrak{e}^{n}_i} + \int_{\markmasspart}\rho^{n}_{\mathbf{z}}\big(\cdot\cap\{\pi\neq\mathbf{1}_n\text{ or }v_1=0\}\big)\,\Lambda(\d\mathbf{z})\Big),
\]
where $\mathfrak{e}^{n}_i\in\markpart{n}$ is defined as
\[
\mathfrak{e}^{n}_i := \Big(\big\{[n]\setminus\{i\},\{i\}\big\},\, (1,\ldots,1,\underbrace{0}_{i-\text{th index}},1,\ldots,1) \Big).
\]
It remains for the first part of the theorem to express $\psi_n$ correctly and to show the integrability condition \eqref{eq:lambda-integrability}.
Recall from \eqref{eq:lambda-n-infty-with-D} that $\lambda_n = \lambda_\infty+\mathcal{D}\big(\{\log v_1\in\cdot\}\cap\{\pi\restr{n} = \mathbf{1}_n\text{ and }v_1\notin \{0,1\}\}\big)$, which shows that
\[
\psi_n(\theta) = id_n\theta -\frac{\beta}{2}\theta^2 + \int_{\markmasspart}\,\sum_{\substack{j\geq 1\\ v_j>0}}s_j^n\left (\e^{i\theta\log v_j} - 1 - i\theta\log v_j \1_{\abs{\log v_j}\leq 1}\right )\,\Lambda(\d \mathbf{z}),
\]
but, from \eqref{eq:betan_dn}, we have
\begin{align*}
d_n &=d_1-\int_{\abs{y}\leq 1}y\,(\lambda_1-\lambda_n)(\d y)\\
&=d_1-\int_{\markmasspart}\,\sum_{\substack{j\geq 1\\ v_j>0}}s_j(1-s_j^{n-1})\log v_j\1_{\abs{\log v_j}\leq 1}\,\Lambda(\d \mathbf{z})
\end{align*}
where we used $(\lambda_1-\lambda_n)= \mathcal{D}\big(\{\log v_1\in\cdot\}\cap\{\pi\restr{n} \neq \mathbf{1}_n\text{ and }v_1\neq 0\}\big)$ which is again deduced from~\eqref{eq:lambda-n-infty-with-D}.
Putting the last two displays together, we get
\[
\psi_n(\theta) = id_1\theta -\frac{\beta}{2}\theta^2 + \int_{\markmasspart}\,\sum_{\substack{j\geq 1\\ v_j>0}}\left (s_j^n\big(\e^{i\theta\log v_j} - 1\big) - i\theta s_j\log v_j \1_{\abs{\log v_j}\leq 1}\right )\,\Lambda(\d \mathbf{z}).
\]
In order to simplify this notation, note that
\begin{align*}
\bigg \lvert \sum_{\substack{j\geq 1\\ v_j>0}} s_j \log{v_j}\1_{\abs{\log v_j}\leq 1}-\log{v_1}\1_{\abs{\log v_1}\leq 1}\bigg\rvert
&\leq (1-s_1\1_{v_1>0})\abs{\log v_1}\1_{\abs{\log v_1}\leq 1} + \sum_{\substack{j\geq 2\\ v_j>0}} s_j\\
& \leq 2(1-s_1\1_{v_1>0}),
\end{align*}
which we proved to be $\Lambda$-integrable in \eqref{eqpr:integ-1-s1}.
Therefore, we can finally define
\[
d := d_1 + \int_{\markmasspart} \bigg( \log{v_1}\1_{\abs{\log v_1}\leq 1}-\sum_{\substack{j\geq 1\\ v_j>0}} s_j \log{v_j}\1_{\abs{\log v_j}\leq 1}\bigg)\, \Lambda(\d \mathbf{z})
\]
in order to get point (i) of Theorem \ref{thm:main}, that is
\[
\psi_n(\theta) = id\theta -\frac{\beta}{2}\theta^2 + \int_{\markmasspart}\bigg(\sum_{\substack{j\geq 1\\ v_j>0}}s_j^n\big(\e^{i\theta\log v_j} - 1\big) - i\theta \log v_1 \1_{\abs{\log v_1}\leq 1}\bigg)\,\Lambda(\d \mathbf{z}).
\]
Now let us show \eqref{eq:lambda-integrability}.
From \eqref{eqpr:integ-1-s1}, it remains only to check that $(\log v_1)^{2}\wedge 1$ is $\Lambda$-integrable.
Since $\lambda_1$ must be a Lévy measure, we have
\begin{align*}
\int_{\R}(y^2\wedge 1) \,\lambda_1(\d y)&=\int_{\markmasspart}\int_{\R}(y^2\wedge 1) \,\rho^{1}_{\mathbf{z}}\big(\{\log v_1\in\d y\}\cap\{v_1\neq 0\}\big)\,\Lambda(\d\mathbf{z})\\
&=\int_{\markmasspart}\sum_{\substack{i\geq 1\\v_i>0}}s_i\big((\log v_i)^2 \wedge 1\big)\,\Lambda(\d\mathbf{z}) < \infty.
\end{align*}
Now note that for all $\mathbf{z}\in\markmasspart$,
\begin{align*}
(\log v_1)^2\wedge 1 &\leq s_1\1_{v_1>0}\big((\log v_1)^2\wedge 1\big) + (1-s_1\1_{v_1>0})\\
&\leq \sum_{\substack{i\geq 1\\v_i>0}}s_i^n\big((\log v_i)^2 \wedge 1\big) + (1-s_1\1_{v_1>0}),
\end{align*}
which is $\Lambda$-integrable.
This proves \eqref{eq:lambda-integrability}, and ends the proof of the main result of Theorem \ref{thm:main}.

For the converse part, let $(c,d,\beta,\Lambda)$ be a given quadruple, where $c,\beta\geq 0$, $d\in\R$, and $\Lambda$ is a measure on $\markmasspart\setminus\{(\mathbf{1},1)\}$ satisfying \eqref{eq:lambda-integrability}.
Then $(\psi_n,J_n,\mathcal{D}_n)$ for all $n\in\N$ are well-defined as in the theorem if one checks that
\begin{equation}\label{tmp:lnbiendef}
\begin{gathered}
\int_{\markmasspart}\,\sum_{\substack{i\geq 1\\ v_i>0}}s_i^n \big((\log v_i)^2\wedge 1\big)\,\Lambda(\d \mathbf{z}) < \infty \\
\text{and}\quad \int_{\markmasspart}\Big (1-\sum_{\substack{i\geq 1\\ v_i>0}}s_i^n\Big )\,\Lambda(\d \mathbf{z}) < \infty.
\end{gathered}
\end{equation}
To that aim, note that for all $\mathbf{z}\in\markmasspart$,
\begin{align*}
\sum_{\substack{i\geq 1\\ v_i>0}}s_i^n \big((\log v_i)^2\wedge 1\big) + \Big (1-\sum_{\substack{i\geq 1\\ v_i>0}}s_i^n\Big ) &\leq s_1^{n}\big((\log v_1)^2\wedge 1\big)+ \sum_{\substack{i\geq 2\\v_i>0}}s_i^n+\Big (1-\sum_{\substack{i\geq 1\\ v_i>0}}s_i^n\Big )\\
&\leq \big((\log v_1)^2\wedge 1\big)+ (1-s_1^{n}\1_{v_1 > 0})\\
&\leq \big((\log v_1)^2\wedge 1\big)+ n(1-s_1\1_{v_1 > 0}),
\end{align*}
which is $\Lambda$-integrable by \eqref{eq:lambda-integrability}, so \eqref{tmp:lnbiendef} is proven.
Now the construction of a $0$-ESSF $X$ with characteristics as above is done via Remark \ref{rk:construction-from-characs}.

\subsection{Proof of Proposition \ref{prop:compactness}} \label{sec:proof-compactness}

Consider a non-degenerate homogeneous ESSF $X=(\Pi,\mathbf{V})$ with characteristics $(c,d,\beta,\Lambda)$.
We make use of a natural genealogy appearing in our construction: jointly for all $n\in\N$, we define processes $(F_n(t),t\geq 0)$ taking values in the subsets of $\N$, starting at $F_n(0)=[n]$, whose role is to ``follow'' integers along a discrete genealogy of blocks.
We will make this statement more precise, but first let us explain the idea.
We will build the $F_n$ deterministically from a sample path of $X$, such that for each time $t\geq 0$,
\[
[n] \subset F_n(t) \subset F_{n+1}(t).
\]
This way, by defining
\begin{equation} \label{eq:def-S-theta-n}
S_\theta^{(n)}(t) = \;\;\sum_{\mathclap{B \text{ block of }X(t)_{|F_n(t)}}}\;\;\tilde{V}_B(t)^{\theta},
\end{equation}
where $\tilde{V}_B(t)$ denotes the mark of a block $B$ of $X(t)$, it is clear, 
since $\bigcup_{n\in\N} F_n (t) = \N$,
that
\[
S_\theta^{(n)}(t) \leq S_{\theta}^{(n+1)}(t) \;\tol_{n\to\infty}\; S_{\theta}(X(t)) = \;\;\sum_{\mathclap{B \text{ block of }X(t)}}\;\;\tilde{V}_B(t)^{\theta}.
\]
The way to define the sets $F_n$ is the following.
Recall that for any $n\in\N$, $T_n$ denotes the first time $t$ when $[n]$ is no longer part of a unique block with positive mark in $X(t)$.
Therefore let $F_n(t) = [n]$ for any $0\leq t < T_n$.
Since $X$ is homogeneous, $T_n$ is an exponential random variable with parameter $J_n$, and conditional on $T_n$, the mark $(V_1(t),0\leq t < T_n)$ behaves as the exponential of a Lévy process $\xi_n$ with characteristic exponent
\[
\psi_n:\theta\mapsto id\theta -\frac{\beta}{2}\theta^2 + \int_{\markmasspart}\bigg(\sum_{\substack{j\geq 1\\ v_j>0}}s_j^n\big(\e^{i\theta\log v_j} - 1\big) - i\theta\log v_1 \1_{\abs{\log v_1}\leq 1}\bigg)\,\Lambda(\d \mathbf{z}).
\]
Then at time $T_n$, $(\Pi(T_n),\mathbf{V}(T_n)/V_1(T_n-))\restr{n}$ is independent of the past and has distribution $\mathcal{D}_n$.
Let us recall that $\mathcal{D}_n$ is expressed in terms of $\Lambda$ by
\[
\mathcal{D}_n=\frac{1}{J_n}\Big(\sum_{i=1}^{n}c\delta_{\mathfrak{e}^{n}_i} + \int_{\markmasspart}\rho^{n}_{\mathbf{z}}\big(\cdot\cap\{\pi\neq\mathbf{1}_n\text{ or }(\pi,\mathbf{v})=(\mathbf{1}_n,0)\}\big)\,\Lambda(\d\mathbf{z})\Big).
\]
At this time, for each block $B$ among the newly created blocks of $X(T_n)$, if $B\cap[n]\neq \emptyset$ and if $B$ has positive mark, let $F_B \subset B$ be the subset consisting of exactly the first $n$ integers that are part of block $B$ (necessarily $B$ contains infinitely many integers so the first $n$ ones exist).
Now we define 
\[
F_n(T_n) :=
[n]
\cup \bigcup_{B}F_{B},
\]
where the union is taken over all newly created blocks of $X(T_n)$ with positive mark and nonempty intersection with
$[n]$.
After this first step, $X(T_n)_{|F_n(T_n)}$ consists of a random but finite (bounded by $n$) number of blocks, those with positive marks containing exactly $n$ integers.
The construction is recursive:
let us write $T_{n,1} := T_n$.
For $k\geq 1$, we define $T_{n,k+1}$ as the next jump time of $\Pi_{|F_n(T_{n,k})}$, that is to say
\[
T_{n,k+1} = \inf\{ t\geq T_{n,k}, \; \Pi(t)_{|F_n(T_{n,k})} \neq \Pi(T_{n,k})_{|F_n(T_{n,k})} \}.
\]
Note that by the branching property, if at time $T_{n,k}$ the marked partition $X(t)_{|F_n(T_{n,k})}$ contains $K$ blocks of size $n$, then $T_{n,k+1}-T_{n,k}$ is an exponential time with parameter $KJ_n$ that is independent of $\sigma(X(t\wedge T_{n,k}), t\geq 0)$.
At time $T_{n,k+1}$, one of the $K$ blocks dislocates, exactly as in the first step.
We can now define the process $F_n$ on the time interval $[T_{n,k},T_{n,k+1}]$: we let $F_n(u)$ constant equal to $F_n(T_{n,k})$ for $u\in [T_{n,k},T_{n,k+1})$, and define
\[
F_n(T_{n,k+1}) = F_n(T_{n,k})\cup \bigcup_{B} F_B,
\] 
where the union is taken over all newly created blocks of $X(T_{n,k+1})$ with positive mark and nonempty intersection with $F_n(T_{n,k+1})$, and as above, $F_B$ denotes the set containing exactly the $n$ smallest elements of block $B$.

This recursive construction defines the process $(F_n(t),t\geq 0)$, such that $(T_{n,k}, k\geq 1)$ is the sequence of its jumping times, and that $X(t)_{|F_n(t)}$ always consists of a finite number of blocks, those with positive mark having size $n$.
Between successive jumping times, the branching property ensures us that each of these blocks has a mark behaving independently as $\e^{\xi_n}$, where $\xi_n$ is a Lévy process with characteristic exponent $\psi_n$.
Therefore if $S_\theta^{(n)}$ is defined by \eqref{eq:def-S-theta-n}, then
\[
(S_\theta^{(n)}(t),\,t\geq 0) \ed \big(\sum_{i}\e^{\theta\xi^{i}_n(t)},\,t\geq 0\big),
\]
where $\big(\xi^{i}_n(t),\,i\geq 1,\,t\geq 0\big)$ is a system of branching particles started from a unique particle at position $0$, which can be described by:
\begin{itemize}
  \item particles move independently as Lévy processes equal to $\xi_n$ in distribution.
  \item a particle branches at rate $J_n$ into a random set of $K$ particles at positions $y+(y_1,\ldots, y_K)$, where $y$ is the position of the mother particle at the time of branching and $(y_1,\ldots, y_K)$ is a vector independent of the past and with distribution given by
  \[
  \E[f(y_1,\ldots,y_K)] = \int_{\markpart{n}} f(\log v_1,\ldots, \log v_K)\,\mathcal{D}_n(\d x),
  \]
  where in the right-hand side integrand, the vector $(v_1,\ldots,v_K)$ denotes the non-zero marks of~$x$.
\end{itemize}
At this point we need the following lemma, which results from standard branching processes arguments.
I could not find a reference which proves this result entirely in this form, so a short, straightforward proof is given below.
\begin{lemma} \label{lem:cumulant-finite-n}
  We have
  \begin{equation} \label{eqpr:cumulant_finite_n}
  \E [S_{\theta}^{(n)}(t)] = \e^{t\kappa^{(n)}(\theta)}, \quad \text{ with } \quad \kappa^{(n)}(\theta) = A^{(n)}(\theta)+J_nB^{(n)}(\theta),
  \end{equation}
  where $A^{(n)}$ corresponds to the movement of particles, with
  \begin{equation} \label{eq:antheta}
  A^{(n)}(\theta) = d\theta +\frac{\beta}{2}\theta^2+\int_{\markmasspart}\,\bigg(\sum_{\substack{i\geq 1\\ v_i>0}}s_i^n\big(v_i^{\theta} - 1\big) - \theta\log v_1 \1_{\abs{\log v_1}\leq 1}\bigg)\,\Lambda(\d \mathbf{z})
  \end{equation}
  and $B^{(n)}$ corresponds to the branching, with
  \begin{align*}
  B^{(n)}(\theta)&=\int_{\markpart{n}}\big(S_\theta(x) - 1\big)\,\mathcal{D}_n(\d x)\\
  &=\frac{1}{J_n}\int_{\markmasspart}\int_{\{\pi\neq\mathbf{1}_n\text{ or }v_1=0\}}\big(S_\theta(x) - 1\big)\,\rho^{n}_{\mathbf{z}}(\d x)\,\Lambda(\d \mathbf{z}).
  \end{align*}
\end{lemma}
\begin{proof}
  It is standard in the theory of Lévy processes (see e.g. \cite[Theorem 25.17]{Sat99}) that $A^{(n)}(\theta) < \infty$ if and only if $\E[\e^{\theta \xi_n(t)}] < \infty$ for all $t\geq 0$, and in that case $\E[\e^{\theta \xi_n(t)}]=\e^{tA^{(n)}(\theta)}$.
  Now fix $0<s<t$ and consider the event
  \[
  A_s^t:=\{\text{the initial particle branches at time $s$ and no other branching occurs before time }t\}.
  \]
  Then conditional on $A_s^t$, the branching construction yields
  \begin{align*}
  \E\Big[\sum_i \e^{\theta\xi_n(t)}\;\Big|\;A_s^t\Big] &= \E\Big[\int_{\markpart{n}}\sum_i \e^{\theta(\xi_n(s)+\log v_i + \widetilde{\xi}_n^{(i)}(t-s))} \,\mathcal{D}_n(\d x)\;\Big|\;A_s^t\Big]\\
  &= \bigg(\int_{\markpart{n}}\sum_i v_i^{\theta}\,\mathcal{D}_n(\d x)\bigg) \E\Big[\e^{\theta(\xi_n(s) + \widetilde{\xi}_n^{(1)}(t-s))}\Big]\\
  &= \bigg(\int_{\markpart{n}}S_\theta(x)\,\mathcal{D}_n(\d x)\bigg) \E\e^{\theta \xi_n(t)}\\
  &= \big(B^{(n)}(\theta)+1\big) \e^{tA^{(n)}(\theta)},
  \end{align*}
  where $\xi_n$ and the $\widetilde{\xi}_n^{(i)},\,i\geq 1$ are i.i.d.\ Lévy processes started from $0$.
  This quantity does not depend on $s$, so one may write, if $A^t$ is the event of a single branching occurring before time $t$,
  \[
  \E\Big[\sum_i \e^{\theta\xi_n^{i}(t)}\;\Big|\;A^t\Big] = \big(B^{(n)}(\theta)+1\big)\e^{tA^{(n)}(\theta)}.
  \]
  and in particular,
  \[
  \E[ S^{(n)}_\theta(t)]=\E\Big[\sum_i \e^{\theta\xi_n^{i}(t)}\Big] \geq \P(A^t)\big(B^{(n)}(\theta)+1\big)\e^{tA^{(n)}(\theta)},
  \]
  which shows that if $A^{(n)}(\theta) = \infty$ or $B^{(n)}(\theta)=\infty$, then $\E[S^{(n)}_\theta(t)]=\infty$.
  Now let us assume that both quantities are finite, and prove \eqref{eqpr:cumulant_finite_n}.
  First note that the argument above readily extends to
  \[
  \E\Big[\sum_i \e^{\theta\xi_n^{i}(t)}\;\Big|\;A^{t,k}\Big] = \big(B^{(n)}(\theta)+1\big)^k \,\E[\e^{\theta\xi_n(t)}].
  \]
  where $A^{t,k}$ is the event of exactly $k$ particles branching before time $t$.
  We now bound from above the probability of $A^{t,k}$.
  Let $t_0:=0<t_1<t_2<\ldots$ denote the branching times in our particle system.
  Note that as particles produce at most $n$ offspring, the time between consecutive branching times $t_j-t_{j-1}$ is greater than an exponential random variable with parameter $nJ_nj$.
  Therefore we can compare the process counting branching times in our process and a Yule process with birth rate $nJ_n$, which yields
  \[
  \P(A^{t,k})\leq \P(t_k < t) \leq (1-\e^{-tnJ_n})^{k}.
  \]
  Now if $t^*$ is small enough so that $(B^{(n)}(\theta)+1)(1-\e^{-tnJ_n}) < 1$ for all $0\leq t\leq t^*$, then we have 
  \[
  \forall 0\leq t \leq t^*,\qquad \E[S^{(n)}_\theta(t)]=\sum_{k\geq 0}\P(A^{t,k})\big(B^{(n)}(\theta)+1\big)^k \,\E[\e^{\theta \xi_n(t)}] < \infty,
  \]
  so the map $f:t\mapsto \E[S^{(n)}_\theta(t)]$ takes finite values before time $t^*$.
  Now note that for any times $t,s\geq 0$ the branching property applied at time $t$ yields $\E[S_\theta(t+s)\mid S_\theta(t)] = S_\theta(t) f(s)$, and so taking expectations, $f(t+s)=f(t)f(s)$.
  Since $f$ takes finite value for $0\leq t\leq t^*$, this shows that $f(t)$ is finite for all $t\geq 0$.
  Now let us compute $f(t)$ by applying the branching property at the first branching time $t_1$:
  \begin{align*}
  f(t) &= \P(t_1 > t) \E[\e^{\theta\xi_n(t)}] + \int_0^{t} \int_{\markpart{n}}\sum_i \E[\e^{\theta(\xi_n(s)+\log v_i)}]f(t-s)\,\mathcal{D}_n(\d x)\, \P(t_1 \in \d s)\\
  &= \e^{-J_nt}\e^{tA^{(n)}(\theta)} + \int_{0}^{t}J_n\e^{-J_n s}\e^{sA^{(n)}(\theta)}\big(B^{(n)}(\theta)+1\big)f(t-s) \,\d s,
  \end{align*}
  and it is easily checked that the only solution of this equation is indeed
  \[
  \E[S^{(n)}_\theta(t)]=f(t)=\e^{t(A^{(n)}(\theta)+J_nB^{(n)}(\theta))}=\e^{t\kappa^{(n)}(\theta)},
  \]
  so \eqref{eqpr:cumulant_finite_n} is proved.
\end{proof}
Now note that for any $\mathbf{z}\in \markmasspart$, one can write
\[
\sum_{\substack{i\geq 1\\ v_i>0}}s_i^n\big(v_i^{\theta} - 1\big) = \int_{\{\pi=\mathbf{1}_n\text{ and }v_1>0\}}\big(S_\theta(x) - 1\big)\,\rho^{n}_{\mathbf{z}}(\d x),
\]
so plugging this into the expression \eqref{eq:antheta} for $A^{(n)}_\theta$ and putting everything together, we have
\[
\kappa^{(n)}(\theta) = d\theta +\frac{\beta}{2}\theta^2+\int_{\markmasspart}\int_{\markpart{n}}\big(S_\theta(x) - 1\big)\,\rho^{n}_{\mathbf{z}}(\d x) - \theta\log v_1 \1_{\abs{\log v_1}\leq 1}\,\Lambda(\d \mathbf{z}).
\]
Now it is a consequence of the law of large numbers that for all $\mathbf{z}\in\markmasspart$, the following convergence holds (and is nondecreasing)
\[
\int_{\markpart{n}}\big(S_\theta(x) - 1\big)\,\rho^{n}_{\mathbf{z}}(\d x) \tol_{n\to\infty} \sum_{i\geq 1}v_i^{\theta} - 1,
\]
and in the end, Lemma \ref{lem:cumulant-finite-n} and monotone convergence yield
\[
\E[S_\theta(X(t))] = \e^{t\kappa(\theta)}.
\]
Now if $\kappa(\theta)$ is finite, it is a simple consequence of the Markov property of the process $X$ that $(\e^{-t\kappa(\theta)}S_{\theta}(X(t)),t\geq 0)$ is a martingale.
Since it is nonnegative it converges almost surely as $t\to\infty$ so it is almost surely bounded by a random variable which we denote by $C=C_\theta>0$.
Now assume furthermore that $\kappa(\theta) < 0$ for some $\theta \neq 0$.
Notice that almost surely for all $i\geq 1$ and $t\geq 0$,
\[
V_i(t)^{\theta} \leq S_{\theta}(X(t)) \leq C \e^{t\kappa(\theta)},
\]
so for any $\alpha\in\R$ such that $-\alpha/\theta >0$,
\[
V_i(t)^{-\alpha} \leq \big(C\e^{t\kappa(\theta)}\big)^{{-\alpha}/{\theta}},
\]
and so almost surely,
\begin{equation} \label{eqproof:time-to-absorption}
\sup_{i\geq 1}\int_{0}^{\infty}V_i(t)^{-\alpha} \,\d t \leq \frac{\theta C^{-\alpha/\theta}}{\alpha\kappa(\theta)} <\infty.
\end{equation}
Now recall the stopping lines
\[
\tau^{-\alpha}_i(t) = \left (\int_{0}^{\cdot}V_i(s)^{-\alpha} \,\d s\right )^{-1}(t), \qquad i\geq 1, \, t\geq 0,
\]
which we used to change the self-similarity index.
Proposition \ref{prop:changing-index} tells us that the time-changed process $X\circ \tau^{-\alpha}$ is an $\alpha$-ESSF process with characteristics $(c,d,\beta,\Lambda)$, and note that for each $i\geq 1$, the integral
\[
\zeta_i = \int_{0}^{\infty}V_i(s)^{-\alpha} \,\d s
\]
is the hitting time of $0$ by the pssMp $V_i \circ \tau_i^{-\alpha}$.
Clearly \eqref{eqproof:time-to-absorption} shows that $X\circ \tau^{-\alpha}$ reaches absorption before time $\frac{\theta C^{-\alpha/\theta}}{\alpha\kappa(\theta)}$.

It remains to show the finite total length property in the case $-\alpha/\theta \geq 1$.
Recall that 
\[
S_{\theta}(X(t)) = \sum_{k\geq 1}\tilde{V}_k(t)^{\theta} \leq C \e^{t\kappa(\theta)}.
\]
It is elementary (because for any summable sequence $u$, $\norm{u}_p\leq \norm{u}_1$ for any $p\geq 1$) that for any $\alpha$ such that $-\alpha/\theta \geq 1$ this implies
\[
S_{-\alpha}(X(t)) = \sum_{k\geq 1}\tilde{V}_k(t)^{-\alpha} \leq \big(C\e^{t\kappa(\theta)}\big)^{{-\alpha}/{\theta}}.
\]
We claim the time change is such that
\[
\int_0^\infty \#X\circ\tau^{-\alpha}(t) \, \d t = \int_0^{\infty} S_{-\alpha}(X(t)) \, \d t \leq \frac{\theta C^{-\alpha/\theta}}{\alpha\kappa(\theta)} < \infty \quad \text{a.s.}
\]
To make this claim entirely justified, let us define for all $x\in\markpartstar$ the (finite of infinite) set $I(x)=\{i_1,i_2,\ldots\}$ where $i_k$ is the first integer contained in the $k$-th block with positive mark of $x$.
Notice that by definition for any $i\geq 1$, for all $t\leq \zeta_i$, $\d \tau_i^{-\alpha}(t) = V_i^{\alpha}(\tau_i^{-\alpha}(t))\d t$, therefore
\begin{align*}
\int_0^\infty \#X\circ\tau^{-\alpha}(t) \, \d t &= \int_0^\infty\quad\sum_{\mathclap{i\in I(X\circ\tau^{-\alpha}(t))}}1\,\d t\\
&= \sum_{i\geq 1} \int_0^\infty\1_{i\in I(X(\tau_i^{-\alpha}(t)))}\frac{V_i^{\alpha}(\tau_i^{-\alpha}(t))\,\d t}{V_i^{\alpha}(\tau_i^{-\alpha}(t))}\\
&= \sum_{i\geq 1} \int_0^\infty\1_{i\in I(X(t))}V_i^{-\alpha}(t)\,\d t\\
&=\int_0^{\infty} S_{-\alpha}(X(t)) \, \d t
\end{align*}
and the proof is complete.

\paragraph{Acknowledgments.}
I thank the \emph{Center for Interdisciplinary Research in Biology} (Collège de France) as well as \emph{Sorbonne Université} for funding, and Amaury Lambert for a careful reading of this manuscript and many helpful comments.


\phantomsection

\end{document}